\documentclass[11pt]{article}

\usepackage[english]{babel}
\usepackage[T1]{fontenc}

\usepackage{amsfonts}
\usepackage{amsmath}
\usepackage{amssymb}
\usepackage{amsthm}
\usepackage{mathtools}
\usepackage{stmaryrd}
\usepackage{dsfont}
\usepackage{graphicx}
\usepackage{hyperref}
\usepackage{titlesec}
\usepackage{mathabx}
\usepackage{enumitem,tocloft}
\usepackage[top=2.5cm, bottom=2.5cm, left=3cm, right=3cm]{geometry}
\usepackage{xcolor}
\usepackage{fourier-orns}
\usepackage[utf8]{inputenc}
\usepackage{bm}
\usepackage{bbm}
\usepackage{framed}
\usepackage{supertabular}
\hypersetup{linktocpage,breaklinks=true}
\usepackage{mathrsfs}
\usepackage{csquotes}
\usepackage{bbold}

\SetSymbolFont{stmry}{bold}{U}{stmry}{m}{n}

\usepackage[backend=bibtex, style=alphabetic, maxbibnames=6, sorting=nyt]{biblatex}

\theoremstyle{plain}
\newtheorem{theorem}{Theorem}[section]
\newtheorem{corollary}[theorem]{Corollary}

\newtheorem{theoremletter}{Theorem}

\newtheorem{corollaryletter}[theoremletter]{Corollary}
\newtheorem{propositionletter}[theoremletter]{Proposition}

\newtheorem{lemma}[theorem]{Lemma}
\newtheorem{proposition}[theorem]{Proposition}

\newtheorem{claim}[theorem]{Claim}
\theoremstyle{definition}
\newtheorem{definition}[theorem]{Definition}

\newtheorem{question}[theorem]{Question}
\newtheorem{remark}[theorem]{Remark}

\numberwithin{equation}{section}

\newcommand{\R}{\mathbb{R}}

\newcommand{\Z}{\mathbb{Z}}	

\newcommand{\N}{\mathbb{N}}	

\newcommand{\shuf}[1]{\mathsf{Shuffler}(#1)}

\makeatletter
\newcommand*{\defeq}{\mathrel{\rlap{%
                     \raisebox{0.3ex}{$\m@th\cdot$}}%
                     \raisebox{-0.3ex}{$\m@th\cdot$}}%
                     =}
\makeatother

\begin{document}

\begin{titlepage}
\setcounter{page}{1}
\title{Quasi-isometric rigidity of lamplighters with lamps of polynomial growth}
\author{\small{Vincent Dumoncel}}

\date{\today}
\maketitle

\begin{abstract}
A quasi-isometry between two connected graphs is measure-scaling if one can control precisely the sizes of pre-images of finite subsets. Such a notion is motivated by the work of Eskin-Fisher-Whyte on lamplighters over $\Z$~\cite{EFW12, EFW13} and the work of Dymarz on biLipschitz equivalences of amenable groups~\cite{Dym05, Dym10}, and led Genevois and Tessera to introduce the scaling group $\text{Sc}(X)$ of an amenable bounded degree graph $X$ in~\cite{GT22}. The main result of our article is a rigidity property for quasi-isometries between lamplighters with lamps of polynomial growth. Under assumptions on $G$ and $H$, any such quasi-isometry $N\wr G\longrightarrow M\wr H$ must be measure-scaling for some scaling factor depending on the growth degrees of $N$ and $M$. In particular, the scaling group of such wreath products is reduced to $\lbrace 1\rbrace$. As applications, we obtain additional examples of pairs of quasi-isometric groups that are not biLipschitz equivalent. We also give applications to the quasi-isometric classification of some iterated wreath products, and we exhibit the first example of an amenable finitely generated group $H$ which is \textit{lamplighter-rigid}, in the sense that $\Z/n\Z\wr H$ and $\Z/m\Z\wr H$ are quasi-isometric if and only if $n=m$.
\end{abstract}

\smallskip

	
	\smallskip
	
	{
		\small	
		\noindent\textbf{{MSC-classification:}}	
		Primary 20F65; Secondary 20F69.
	}


\tableofcontents

\section{Introduction}\label{section1}

\smallskip

In geometric group theory, a challenging question is to understand the collection of all maps between two given metric spaces that are compatible with the large-scale geometry of our spaces. Such collections encompass for instance quasi-isometries, coarse and quasi-isometric embeddings, or regular maps. The motivation behind this program is that, on the one hand finitely generated groups are naturally metric spaces and can thus be studied from a geometric point of view, and on the other hand their geometric properties are closely related to their algebraic structure. 

\smallskip

\sloppy Several milestones have been achieved in the study of the quasi-isometric rigidity of many classes of groups, among which abelian and non-abelian free groups~\cite{Dun85}, mapping class groups~\cite{Beh+12}, non-uniform lattices of $\text{Isom}(\mathbb{H}^{n})$~\cite{Sch96}, $n\ge 3$, lamplighters over $\Z$ and others $\text{SOL}-$like groups~\cite{EFW12, EFW13}, or Baumslag-Solitar groups~\cite{FM98, FM99, Why01}.

\smallskip

In a recent work~\cite{GT24b}, Genevois and Tessera established a complete classification up to quasi-isometry of lamplighters over finitely presented one-ended groups. Precisely, they showed that the existence of a quasi-isometry between $E\wr H$ and $F\wr K$ is parametrized by the existence of a quasi-isometry between $H$ and $K$ which is compatible in a certain way with the lamp groups $E$ and $F$, in the sense that it must be measure-scaling for some scaling factor that depends on the cardinalities of $E$ and $F$. Roughly speaking, a quasi-isometry $X\longrightarrow Y$ between two metric spaces is quasi-$k$-to-one if the size of the pre-image of a finite subset $A$ of $Y$ is $k\left|A\right|$, up to an error controlled uniformly over $A\subset Y$ (precise definitions are given in Section ~\ref{section2}). This quantification is also motivated by a famous result of Whyte (recalled in Theorem~\ref{thm:whytetheorem}), which says that being quasi-one-to-one is the same as lying within bounded distance from a bijection. Furthermore, the property of being measure-scaling is well-behaved under composition and quasi-inverses (see Theorem~\ref{thm:stabilitypropertiesofscalingfactors}), which yields to define the scaling group of an amenable space $X$ as
\begin{equation*}
    \text{Sc}(X) \defeq \left\lbrace k>0 : \exists \; \text{a quasi-$k$-to-one quasi-isometry $X\longrightarrow X$}\right\rbrace. 
\end{equation*}

\sloppy In this way, $\text{Sc}(X)$ is a subgroup of $(\R_{>0},\cdot)$. Given an amenable space $X$, its scaling group is an object of independent interest, that encodes many possible behaviours for its self-quasi-isometries. In the case of $X=G$ a finitely generated group, identifying its scaling group can give access to informations on its algebraic structure. For instance, it is not hard to check that, given a finite-index subgroup $H$ of $G$, the natural inclusion $H\hookrightarrow G$ is quasi-$\frac{1}{[G:H]}$-to-one. Thus, if $G$ has trivial scaling group, it cannot contain proper finite index subgroups isomorphic to itself.

\smallskip

Scaling groups of several classes of amenable groups, among which Carnot groups and Baumslag-Solitar groups, have been computed in~\cite[Corollary~6.6]{GT22}. Additionally, as an important consequence of the classification established in~\cite[Theorem~1.4]{GT24b}, \mbox{$\text{Sc}(F\wr K)=\lbrace 1\rbrace$} if $F$ is finite and $K$ is amenable, finitely presented and one-ended. This observation heavily relies on the general form of quasi-isometries between lamplighters, which are shown in~\cite{GT24b} to all preserve the lamplighter structure in the following way:
\begin{definition}
Let $M,N,G,H$ be finitely generated groups. A quasi-isometry $q\colon N\wr G\longrightarrow M\wr H$ is \textit{of aptolic form} if there exist two maps $\alpha \colon N^{(G)} \longrightarrow M^{(H)}$ and $\beta\colon G\longrightarrow H$ such that 
\begin{equation*}
    q(c,p)=(\alpha(c),\beta(p))
\end{equation*}
for any $(c,p)\in N\wr G$. A quasi-isometry $q\colon N\wr G\longrightarrow M\wr H$ is \textit{aptolic} if it is of aptolic form and has a quasi-inverse of aptolic form. 
\end{definition}

Precisely,~\cite[Theorem~1.18]{GT24b} says that any quasi-isometry between two lamplighters $F_{1}\wr K_{1}$ and $F_{2}\wr K_{2}$, where $K_{1},K_{2}$ are amenable finitely presented and one-ended, is at a bounded distance from an aptolic quasi-isometry $F_{1}\wr K_{1}\longrightarrow F_{2}\wr K_{2}$. The authors then deduce from this rigidity phenomenon that any quasi-isometry between $F_{1}\wr K_{1}$ and $F_{2}\wr K_{2}$ is measure-scaling for a scaling factor depending on $|F_{1}|$ and $|F_{2}|$.

\smallskip

Here we show that, for lamplighters with lamps of polynomial growth, if all quasi-isometries are aptolic (up to bounded distances), then they are also all measure-scaling, with a scaling factor depending on the involved growth degrees. 

\begin{theoremletter}\label{thm:maintheorem}
Let $N$ and $M$ be finitely generated groups of polynomial growth, with growth degrees $n$ and $m$ respectively. Let $G$ and $H$ be finitely generated amenable groups such that any quasi-isometry between $N\wr G$ and $M\wr H$ is (up to bounded distance) aptolic. Then any quasi-isometry $N\wr G\longrightarrow M\wr H$ is quasi-$\frac{m}{n}$-to-one.
\end{theoremletter}

In~\cite{BGT24} are provided explicit examples of groups $G$ and $H$ for which any quasi-isometry $N\wr G\longrightarrow M\wr H$ is aptolic, up to a finite distance change. Thus:

\begin{theoremletter}\label{thm:maintheoremfortheclassMexp}
Let $N$ and $M$ be finitely generated groups of polynomial growth, with growth degrees $n$ and $m$ respectively. Let $G$ and $H$ be finitely presented amenable groups from $\mathcal{M}_{\text{exp}}$. Then any quasi-isometry $N\wr G \longrightarrow M\wr H$ is quasi-$\frac{m}{n}$-to-one.  
\end{theoremletter}

We refer to Section~\ref{section2} for the precise definition of the class $\mathcal{M}_{\text{exp}}$. For now, let us simply mention that it contains many amenable groups: solvable Baumslag-Solitar groups $\text{BS}(1,n)$, $n\ge 2$, $\text{SOL}(\Z)$, lamplighter groups, as well as any direct product of one of these groups with an arbitrary finitely generated amenable group.

\smallskip

We also emphasize that, in these theorems, the amenability condition is necessary to avoid an empty statement: between non-amenable spaces, a quasi-isometry is measure-scaling for all possible scaling factors.

\smallskip

\noindent \sloppy \textbf{Quasi-isometries and biLipschitz equivalences.} Refining our understanding of quasi-isometries, an other interesting problem is to understand to what extent being quasi-isometric and being biLipschitz equivalent differ. Whyte's theorem mentioned above ensures that, for non-amenable spaces, a quasi-isometry can always be turned into a biLipschitz equivalence by a finite distance change. For amenable spaces, the question is much more subtle, and no analog of Whyte's theorem can hold: indeed, for a proper finite-index subgroup $H$ of an amenable finitely generated group $G$, the inclusion $H\hookrightarrow G$ is quasi-$\frac{1}{[G:H]}$-to-one and thus does not lie at bounded distance from a bijection. The first example of a pair of amenable quasi-isometric groups that are not biLipschitz equivalent appears in~\cite{Dym10} and involves lamplighters over $\Z$. 

\smallskip

Concerning biLipschitz equivalences between lamplighters with infinite lamp groups, we can deduce from our main result the next consequence:

\begin{corollaryletter}\label{cor:BLE}
Let $N$ and $M$ be finitely generated groups with polynomial growth. Let $G$ and $H$ be finitely presented amenable groups from $\mathcal{M}_{\text{exp}}$. If $N\wr G$ and $M\wr H$ are biLipschitz equivalent, then $N$ and $M$ have same growth degrees and $G$ and $H$ are biLipschitz equivalent. 
\end{corollaryletter}

It is worth noticing that the converse statement is not true: according to~\cite[Corollary~6.36]{BGT24}, if $N$ and $M$ are nilpotent and $N\wr G$ and $M\wr H$ are quasi-isometric, then $N$ and $M$ must have the same nilpotent step. Thus, for instance, $\Z^4\wr\text{BS}(1,n)$ and $\text{Heis}(\Z)\wr\text{BS}(1,n)$ are not quasi-isometric, even though $\Z^4$ and $\text{Heis}(\Z)$ have equal growth degrees. Restricting to virtually abelian groups, we can completely describe the quasi-isometry class of our wreath products. 

\begin{theoremletter}\label{thm:classificationforvirtuallyabeliangroups}
Let $A_{1}$ and $A_{2}$ be infinite virtually abelian finitely generated groups, with growth degrees $d_{1}$ and $d_{2}$ respectively. Let $G$ and $H$ be finitely presented amenable groups from $\mathcal{M}_{\text{exp}}$. Then:
\begin{enumerate}[label=(\roman*)]
    \item $A_{1}\wr G$ and $A_{2}\wr H$ are quasi-isometric if and only if there exists a quasi-$\frac{d_{2}}{d_{1}}$-to-one quasi-isometry $G\longrightarrow H$;
    \item $A_{1}\wr G$ and $A_{2}\wr H$ are biLipschitz equivalent if and only if $d_{1}=d_{2}$ and there exists a biLipschitz equivalence $G\longrightarrow H$.
\end{enumerate}
\end{theoremletter}

We thus get additional examples of pairs of quasi-isometric groups that are not biLipschitz equivalent: for instance, as $\text{Sc}(\text{BS}(1,k))=\R_{>0}$, $\Z^{n}\wr\text{BS}(1,k)$ and $\Z^{m}\wr\text{BS}(1,k)$ are quasi-isometric for all $n,m\ge 1$ and $k\ge 2$, while they are biLipschitz equivalent if and only if $n=m$. 

\smallskip

In the opposite direction, as mentioned above, Genevois and Tessera showed in~\cite{GT24b} that a lamplighter group $F\wr K$, where $F$ is finite and $K$ is amenable, one-ended and finitely presented, has all its self-quasi-isometries at bounded distance from bijections. As a consequence of Theorem~\ref{thm:maintheoremfortheclassMexp}, the same phenomenon occurs for lamplighters with lamps of polynomial growth:

\begin{corollaryletter}
Let $N$ and $M$ be finitely generated groups with polynomial growth. Let $G$ and $H$ be finitely presented amenable groups from $\mathcal{M}_{\text{exp}}$. If $N$ and $M$ have equal growth degrees, then any quasi-isometry $N\wr G\longrightarrow M\wr H$ lies within bounded distance from a bijection. 

In particular, any quasi-isometry $N\wr G\longrightarrow N\wr G$ lies at bounded distance from a bijection, and thus $\text{Sc}(N\wr G)=\lbrace 1\rbrace$.
\end{corollaryletter}

As mentioned above, the fact that $N\wr G$ has a trivial scaling group allows us to deduce some algebraic facts on its subgroups of finite index. For instance, two biLipschitz equivalent finite-index subgroups of $N\wr G$ must have the same index, and, in particular, $N\wr G$ does not have any proper finite-index subgroup isomorphic to itself, a fact that is not obvious to prove from a purely algebraic point of view. In turn, this has the consequence that any non-surjective monomorphism $N\wr G \hookrightarrow N\wr G$ has an image of infinite index (such morphisms exist as soon as $N$ or $G$ is not co-Hopfian, e.g. $N=\Z^{n}$ or $G=\text{BS}(1,n)$ for $n\ge 1$; see~\cite[Theorem~6.1]{BFF24}). 
\vspace{0.15cm}

\noindent \textbf{Application to lamplighter-rigidity.} So far, related to the quasi-isometric classification of wreath products of the form $(\text{finite})\wr(\text{finitely generated})$, all results recorded in the literature exhibits a flexibility part, namely given a finitely generated group $H$ and integers $n,m\ge 2$ satisfying some arithmetic condition, it is often possible to construct a quasi-isometry between $\Z_{n}\wr H$ and $\Z_{m}\wr H$. This observation then suggests the next question:
\begin{question}\label{question1.5}
Does there exist a finitely generated group $H$ such that $\Z_{n}\wr H$ and $\Z_{m}\wr H$ are never quasi-isometric if $n\neq m$?
\end{question}

We know the answer when $H$ belongs to various classes of groups: for instance, if $H$ is non-amenable, $\Z_{n}\wr H$ and $\Z_{m}\wr H$ are quasi-isometric as soon as $n$ and $m$ have the same prime divisors, by~\cite[Theorem~3.12]{GT24b}. Hence an example of a group $H$ must necessarily come from the class of amenable groups.

\smallskip

But, further, if $H$ is amenable, it has finitely many ends, and thus it has either zero, one or two ends. The zero end case is trivial, as then both $\Z_{n}\wr H$ and $\Z_{m}\wr H$ are finite and thus quasi-isometric, for any choice of $n$ and $m$. The two ended case is also completely understood by the work of Eskin-Fisher-Whyte~\cite{EFW12, EFW13}, and in this case it suffices to take $n$ and $m$ to be powers of a common number to construct a quasi-isometry $\Z_{n}\wr H\longrightarrow \Z_{m}\wr H$. Therefore, for a positive answer to Question~\ref{question1.5}, one has to look into amenable one-ended groups. In this class, the case of finitely presented groups is also completely understood, and here as well it suffices to take $n=a^{s}$ and $m=a^{r}$ powers of a common number to deduce that $\Z_{n}\wr H$ and $\Z_{m}\wr H$ are quasi-isometric if $\frac{s}{r}\in\text{Sc}(H)$~\cite[Theorem~3.11]{GT24b}. Hence, within this class, an example of a group $H$ answering positively Question~\ref{question1.5} should have trivial scaling group. However, we do not know a single example of a finitely presented one-ended group with trivial scaling group. 

\smallskip

On the other hand, despite the fact that our wreath products $N\wr G$ are not finitely presented, they do have trivial scaling group, and we show that they indeed provide examples of groups answering positively Question~\ref{question1.5}:

\begin{propositionletter}\label{prop:lamplighterrigidity}
Let $n,m\ge 2$ be two integers. Let $N$ be a finitely generated group with polynomial growth, and let $G$ be a finitely presented amenable group from $\mathcal{M}_{\text{exp}}$. Then the lamplighters $\Z_{n}\wr(N\wr G)$ and $\Z_{m}\wr(N\wr G)$ are quasi-isometric if and only if $n=m$. 
\end{propositionletter}

More generally, we know from~\cite{GT24b} that the existence of a quasi-isometry \mbox{$\Z_{n}\wr G\longrightarrow \Z_{m}\wr H$} provides a measure-scaling quasi-isometry $G\longrightarrow H$, where the scaling factor depends on $n$ and $m$. Thus, when $G$ and $H$ are themselves lamplighters with lamps of polynomial growth, the existence of such a quasi-isometry imposes a compatibility condition between $n$, $m$ and the growth degrees of lamp groups. This is the content of the next statement.

\begin{propositionletter}\label{prop:mixingofscalingconditions}
Let $n,m\ge 2$ be two integers. Let $N_{1}$ and $N_{2}$ be finitely generated groups of polynomial growth, with growth degrees $n_{1}$ and $n_{2}$ respectively. Let $G$ and $H$ be finitely presented amenable groups from $\mathcal{M}_{\text{exp}}$. If \;$\Z_{n}\wr(N_{1}\wr G)$ and $\Z_{m}\wr(N_{2}\wr H)$ are quasi-isometric, then there exist $a,r,s\ge 1$ such that $n=a^{r}$, $m=a^{s}$, and $\frac{s}{r}=\frac{n_{2}}{n_{1}}$. 
\end{propositionletter}

\noindent \textbf{Iterated wreath products.} In the light of Proposition~\ref{prop:mixingofscalingconditions}, it is also natural to try to fully classify iterated wreath products. For such groups, most of the technology developed in~\cite{GT24a, GT24b} do not apply, because lamplighter groups $F\wr K$ with $F$ finite do not have the \textit{thick bigon property}, a key quasi-isometry invariant introduced in~\cite{GT24a}. Roughly speaking, a finitely generated group $G$ has the thick bigon property if given any two points $x,y\in G$ connected by a path $\gamma_{1}$ and any point $p\in \gamma_{1}$ far enough from $x$ and $y$, it is always possible to connect $x$ and $y$ by a path $\gamma_{2}$ which is coarsely homotopic to $\gamma_{1}$ and that avoids arbitrary large balls centered at $p$ (see~\cite[Definition~3.2]{GT24a} for a more precise definition). It turns out that, as soon as $F$ is infinite, wreath products $F\wr K$ have this property, and building on Theorem~\ref{thm:classificationforvirtuallyabeliangroups}, we get a complete classification of wreath products of the form 
\begin{equation*}
    (\text{finite})\wr((\text{infinite virtually abelian})\wr(\text{finitely presented from $\mathcal{M}_{\text{exp}}$})).
\end{equation*}

\begin{corollaryletter}\label{cor:classificationofiteratedwreathproducts}
Let $n,m\ge 2$ be two integers. Let $A_{1}$ and $A_{2}$ be infinite virtually abelian finitely generated groups, with growth degrees $d_{1}$ and $d_{2}$ respectively. Let $G$ and $H$ be finitely presented amenable groups from $\mathcal{M}_{\text{exp}}$. Then:
\begin{enumerate}[label=(\roman*)]
    \item $\Z_{n}\wr(A_{1}\wr G)$ and $\Z_{m}\wr(A_{2}\wr H)$ are quasi-isometric if and only if there exist $a,r,s\ge 1$ such that $n=a^{r}$, $m=a^{s}$, $\frac{s}{r}=\frac{d_{2}}{d_{1}}$, and there exists a quasi-$\frac{d_{2}}{d_{1}}$-to-one quasi-isometry \;$G\longrightarrow H$;
    \item $\Z_{n}\wr(A_{1}\wr G)$ and $\Z_{m}\wr(A_{2}\wr H)$ are biLipschitz equivalent if and only if $n=m$, $d_{1}=d_{2}$ and there exists a biLipschitz equivalence \;$G\longrightarrow H$.
\end{enumerate}
\end{corollaryletter}

For instance, this corollary rules out the existence of a quasi-isometry 
\begin{equation*}
\Z_{2}\wr(\Z^{2}\wr \text{BS}(1,n)) \longrightarrow \Z_{4}\wr(\Z^{3}\wr \text{BS}(1,n))
\end{equation*}
where $n\ge 2$.

\smallskip

It is worth noticing that, if the amenability assumption on $G$ and $H$ is removed, a classification can also be easily deduced from \cite{BGT24} and Proposition~\ref{prop:constructionsofQI} below: $\Z_{n}\wr(A_{1}\wr G)$ and $\Z_{m}\wr(A_{2}\wr H)$ are quasi-isometric if and only if $n$ and $m$ have the same prime divisors and $G$ and $H$ are quasi-isometric.

\smallskip

\noindent \textbf{More wreath products.} Observe that, when comparing two wreath products over the same bases, many quasi-isometry invariants (such as the volume growth, the number of ends, F\o lner functions, divergence, asymptotic dimension to name a few) fail to distinguish them. In general, much more involved methods, as the ones developed in~\cite{EFW12, EFW13, GT24a, GT24b}, appear necessary to exhibit a significative geometric difference between our groups. Those strategies require additional assumptions, relying on notions such as the thick bigon property in~\cite{GT24a} or coarse separation in~\cite{BGT24}. We conclude the article by noticing that our construction of quasi-isometries, combined with results of Erschler~\cite{Dyu00, Ers03}, also provide classification results in situations where techniques from~\cite{GT24a, BGT24} cannot be applied.  

\begin{corollaryletter}\label{cor:classificationforvirtuallyabeliangroups2}
Let $A_{1},A_{2},B_{1},B_{2}$ be infinite virtually abelian finitely generated groups. Then $A_{1}\wr B_{1}$ and $A_{2}\wr B_{2}$ are quasi-isometric if and only if $B_{1}$ and $B_{2}$ have same growth degrees. 
\end{corollaryletter}

For instance, $\Z^d\wr\Z^{k}$ and $\Z^{d'}\wr\Z^{k'}$ are quasi-isometric if and only if $k=k'$.

\medskip

\noindent \textbf{Plan of the paper.} Section~\ref{section2} introduces in detail the necessary background on scaling quasi-isometries, and Section~\ref{section3} the necessary results on aptolic quasi-isometries between wreath products. Section~\ref{section4} is dedicated to the proof of our main result. Section~\ref{section5} presents the various mentioned applications, and Section~\ref{section6} records several questions related to the article. 

\medskip

\noindent \textbf{Acknowledgements.} I am grateful to Anthony Genevois for useful dicussions, and to Romain Tessera for having encouraged me to investigate those topics. I also thank Corentin Correia and Juan Paucar for useful comments on an earlier version of the article. 
\section{Coarse geometry and scaling quasi-isometries}\label{section2}

\subsection{Notations}\label{subsection2.1} 
Given a group $G$, we denote $1_{G}$ its neutral element, and if $S\subset G$ is a finite generating set of $G$, we identify $G$ with its Cayley graph $\text{Cay}(G,S)$ when this does not cause ambiguity. For an integer $n\ge 1$, $\Z_{n}$ denotes the cyclic group of order $n$.

\smallskip

Given a metric space $(X,d_{X})$, we denote $B_{X}(x,r)$ the closed ball centered at $x\in X$ of radius $r>0$, that is the set $\lbrace y\in X : d_{X}(x,y)\le r\rbrace$. 

\smallskip

In the sequel, $\simeq$ denotes the usual asymptotic equivalence relation on non-decreasing functions $\N\longrightarrow \N$: $f\simeq g$ if there exist two constants $C,C'>0$ such that $f(n)\le Cg(Cn)$ for all $n\in\N$ and $g(n)\le C'\cdot f(C'n)$ for all $n\in\N$. 

\smallskip

Any graph in this text is unoriented and simplicial, that is without loops nor multiple edges. For such a graph $\Gamma$, $V(\Gamma)$ refers to its set of vertices, while $E(\Gamma)$ stands for the set of edges. Given any subset $A\subset\Gamma$, its boundary in $\Gamma$ is denoted $\partial_{\Gamma}A$ and is defined as
\begin{equation*}
    \partial_{\Gamma}A \defeq \left\lbrace y\in \Gamma\setminus A : \exists x\in A, (x,y)\in E(\Gamma)\right\rbrace.
\end{equation*}

\subsection{Coarse geometry}\label{subsection2.2} Given two metric spaces $(X,d_{X}), (Y,d_{Y})$ and constants $C\ge 1$, $K\ge 0$, a map $f\colon (X,d_{X})\longrightarrow (Y,d_{Y})$ is a \textit{$(C,K)-$quasi-isometry} if:
\begin{itemize}
    \item $\frac{1}{C}\cdot d_{X}(x,y)-K\le d_{Y}(f(x),f(y)) \le C\cdot d_{X}(x,y)+K$ for any $x,y\in X$;
    \item $d_{Y}(y, f(X)) \le K$ for any $y\in Y$.
\end{itemize}
Equivalently, $f\colon X\longrightarrow Y$ is a $(C,K)-$quasi-isometry if it satisfies the first condition and if there exists a map $g\colon Y\longrightarrow X$ such that $d(g\circ f, \text{Id}_{X})\le K$ and $d(f\circ g, \text{Id}_{Y})\le K$, where the distance between two maps $h_{1},h_{2}\colon X\longrightarrow Y$ defined on the same metric space is
\begin{equation*}
    d(h_{1},h_{2})\defeq \sup_{x\in X}d_{Y}(h_{1}(x),h_{2}(x)).
\end{equation*}
The map $g$ is called a \textit{quasi-inverse} of $f$, and it is unique up to bounded distances. A $(C,0)-$quasi-isometry $f\colon X\longrightarrow Y$ is usually called a \textit{biLipschitz equivalence}, or a \textit{$C-$biLipschitz equivalence}, and if $f$ satisfies only
\begin{equation*}
    d_{Y}(f(x), f(y)) \le Cd_{X}(x,y)
\end{equation*}
for any $x,y\in X$, we call it \textit{$C-$Lipschitz}. 

\smallskip

The next easy observation often simplifies computations for checking that a given map is Lipschitz.

\begin{lemma}\label{lm:lipschitzmap}
Let $f\colon X\longrightarrow Y$ be a map between two bounded degree graphs. If there exists $C>0$ such that $d_{Y}(f(x),f(y)) \le C$ for any pair $(x,y)$ of adjacent vertices of $X$, then $f$ is $C-$Lipschitz.
\end{lemma}

\begin{proof}
Let $x,y\in X$ and pick $x=x_{0},x_{1},\dots,x_{n-1},x_{n}=y$ a geodesic between $x$ and $y$. Then, from the triangle inequality and the assumption, one gets
\begin{equation*}
    d_{Y}(f(x),f(y)) \le \sum_{i=0}^{n-1}d_{Y}(f(x_{i}), f(x_{i+1})) \le C\cdot n=C\cdot d_{X}(x,y)
\end{equation*}
and thus $f$ is $C-$Lipschitz. 
\end{proof}

Recall also that, in a metric space $(X,d_{X})$, given a number $S\ge 0$ and a subset $A\subset X$, its \textit{$S-$neighborhood} is denoted $A^{+S}$ and is defined as 
\begin{equation*}
    A^{+S}\defeq \bigcup_{a\in A}B(a,S).
\end{equation*}
Then, the \textit{Hausdorff distance} between two subsets $A,B\subset X$ is 
\begin{equation*}
    d_{\text{Haus}}(A,B)\defeq \inf\lbrace S\ge 0 : A\subset B^{+S}, B\subset A^{+S}\rbrace. 
\end{equation*}

We will need the following facts in our computations.

\begin{lemma}\label{lm:cardinalityofneighborhood}
Let $X$ be a bounded degree graph. Then one has $|A^{+S}|\le N^{S}\cdot |A|$ for any finite subset $A\subset X$ and any $S\ge 0$, where $N\ge 3$ is an integer larger than the maximal degree of a vertex of $X$.
\end{lemma}

\begin{proof}
Fix any $0\le i\le S-1$. Given any element $a$ of $A$, the fact that any vertex of $X$ has degree $\le N$ implies that there are at most $N^{i}$ paths of length $i$ that starts at $a\in A$ and that ends at an element of $A^{+i}$. Thus 
\begin{equation*}
    \left|A^{+S}\right|\le \left|A\right|\sum_{i=0}^{S-1}N^{i} \le N^{S}\left|A\right|
\end{equation*}
as claimed. 
\end{proof}

\begin{lemma}\label{lm:cardinalityofneighborhood2}
Let $X$ be a bounded degree graph. For any finite subset $A\subset X$ and any $S\ge 0$, there is a constant $R>0$ (depending only on $S$ and on $X$) such that $|A^{+S}\setminus A| \le R\cdot |\partial_{X} A|$.
\end{lemma}

\begin{proof}
It suffices to note that $A^{+S}\setminus A \subset (\partial_{X}A)^{+(S-1)}$ and to apply Lemma~\ref{lm:cardinalityofneighborhood}. 
\end{proof}

\begin{lemma}\label{lm:preimagesandquasiinverses}
Let $X$ and $Y$ be bounded degree graphs, and let $f\colon X\longrightarrow Y$ be a $(C,K)-$quasi-isometry, with a $(C,K)-$quasi-inverse $\overline{f}\colon Y\longrightarrow X$. Then we have 
\begin{equation*}
    d_{\text{Haus}}\left(\overline{f}(A),f^{-1}(A^{+K})\right)\le (C+2)K
\end{equation*}
for any subset $A\subset Y$.
\end{lemma}

\begin{proof}
Let $A\subset Y$, and let $x\in f^{-1}(A^{+K})$. Then $f(x)\in A^{+K}$, and we may pick $y\in A$ such that $d_{Y}(y, f(x))\le K$. It follows that 
\begin{align*}
    d_{X}(\overline{f}(y), x) &\le d_{X}(\overline{f}(y), \overline{f}(f(x)))+d_{X}(\overline{f}(f(x)), x) \\
    &\le C\cdot d_{Y}(y, f(x))+K+K \\
    &\le (C+2)K
\end{align*}
and since $\overline{f}(y)\in \overline{f}(A)$, we get $x\in \overline{f}(A)^{+(C+2)K}$. 

\smallskip

Conversely, if $x=\overline{f}(y)$ with $y\in A$, then $f(x)=f(\overline{f}(y))$ is at distance at most $K$ from $y\in A$, i.e. $f(x)\in A^{+K}$, so $x\in f^{-1}(A^{+K})$. We conclude that $d_{\text{Haus}}\left(\overline{f}(A),f^{-1}(A^{+K})\right) \le (C+2)K$, as claimed. 
\end{proof}

We now define coarse separation of spaces, as in~\cite{BGT24}. For this, recall that a metric space $(X,d_{X})$ is \textit{$k-$coarsely connected}, with $k>0$, if for any $x,y\in X$, there is a sequence of points $x=x_{0},x_{1},\dots,x_{n-1},x_{n}=y$ such that $d_{X}(x_{i-1},x_{i})\le k$ for any $i=1,\dots, n$. 

\begin{definition}
Let $(X,d_{X})$ be a metric space, and let $\mathcal{W}$ be a family of subsets of $X$. We say that $\mathcal{W}$ \textit{coarsely separates} $X$ if there exist $k>0$ and $L\ge 0$ such that for any $D\ge 0$, there exists $W\in\mathcal{W}$ such that $X\setminus W^{+L}$ has at least two $k-$coarsely connected components with points at distance $\ge D$ from $W$.
\end{definition}

Given a family $\mathcal{W}$ of subsets of $X$, we define its \textit{growth function} as 
\begin{equation*}
    V_{\mathcal{W}}(r)=\sup_{W\in\mathcal{W}, \;w\in W}|W\cap B(w,r)|, \; r\ge 0
\end{equation*}
and we say that $\mathcal{W}$ has \textit{subexponential growth} if $\limsup_{r\rightarrow\infty}\frac{1}{r}\ln(V_{\mathcal{W}}(r))=0$. We denote $\mathcal{M}_{\text{exp}}$ the class of bounded degree graphs such that any coarsely separating family of subsets must have exponential growth. Note that this class is invariant under quasi-isometries.

\smallskip

The following theorem is one of the main results of~\cite{BGT24}. 

\begin{theorem}[{\cite[Theorem~1.4]{BGT24}}]
For $i=1,2$, let $X_{i}$ be either a $(k+1)-$regular tree with $k\ge 2$ or a rank one symmetric space of non-compact type. Then the horocyclic product $X_{1}\bowtie X_{2}$ belongs to $\mathcal{M}_{\text{exp}}$, as well as any direct product of $X_{1}\bowtie X_{2}$ with an arbitrary bounded degree graph. 
\end{theorem}

As already mentioned, this theorem provides many examples of amenable groups that belong to $\mathcal{M}_{\text{exp}}$, among which solvable Baumslag-Solitar groups $\text{BS}(1,n)$, $n\ge 2$, lamplighter groups and $\text{SOL}(\Z)$. 

\subsection{Measure-scaling quasi-isometries} The goal of this section is to recall the notion of 
measure-scaling quasi-isometries between two bounded degree graphs introduced in~\cite{GT22}, and useful properties of such maps. 

\begin{definition}
Let $X$ and $Y$ be two bounded degree graphs, and let $f\colon X\longrightarrow Y$ be a quasi-isometry. Let $k>0$. We say that $f$ is \textit{quasi-$k$-to-one} if there exists a constant $C>0$ such that 
\begin{equation*}
    \left|k|A|-|f^{-1}(A)|\right| \le C\cdot\left|\partial_{Y} A\right|
\end{equation*}
for any finite subset $A\subset Y$.
\end{definition}

In this case, we call $f$ a \textit{measure-scaling quasi-isometry}, and we refer to the number $k>0$ as the \textit{scaling factor}. The idea to keep in mind is that the scaling factor quantifies how far is $f$ from being a bijection. This quantification is motivated by the following famous result of Whyte:

\begin{theorem}[{\cite[Theorem~2]{Why99}}]\label{thm:whytetheorem}
A quasi-isometry between two bounded degree graphs is quasi-one-to-one if and only if it lies at finite distance from a bijection. 
\end{theorem}

It is easy to see that, if $X$ is not amenable, then any quasi-isometry from $X$ to $Y$ is quasi-$k$-to-one for any $k>0$. In particular, if $X$ is not amenable, any quasi-isometry $X\longrightarrow Y$ lies within bounded distance from a bijection. This behaviour contrasts widely with the amenable case, for two reasons. The first one is that a quasi-isometry between two amenable graphs is not necessarily scaling: for instance the quasi-isometry $g\colon \Z\longrightarrow \Z$ given by $g(n)=n$ if $n\ge 0$ and $g(n)=2n$ if $n<0$ is not measure-scaling. The second reason is that, if a quasi-isometry is scaling, the scaling factor is unique:

\begin{lemma}\label{lm:uniquenessofscalingfactors}
Let $X$ and $Y$ be two bounded degree graphs, and let $f\colon X\longrightarrow Y$ be a quasi-isometry. Suppose that $X$ is amenable. If $f$ is quasi-$k$-to-one and quasi-$k'$-to-one, then $k=k'$. 
\end{lemma}

\begin{proof}
Let $(F_{n})_{n\in\N}$ be a F\o lner sequence for $X$, so that $\frac{\left|\partial_{X}F_{n}\right|}{\left|F_{n}\right|}$ goes to $0$ as $n\rightarrow\infty$. Since $f$ is quasi-$k$-to-one, there is $C>0$ such that 
\begin{equation*}
    \left|k-\frac{\left|f^{-1}(F_{n})\right|}{\left|F_{n}\right|}\right| \le C\cdot \frac{\left|\partial_{X}F_{n}\right|}{\left|F_{n}\right|}
\end{equation*}
for any $n\in\N$. In particular, $\left(\frac{\left|f^{-1}(F_{n})\right|}{\left|F_{n}\right|}\right)_{n\in\N}$ converges to $k$ as $n\rightarrow\infty$, and similarly it also converges to $k'$. Thus $k=k'$.
\end{proof}

The next statement, proved in~\cite{GT22}, records several stability properties for measure-scaling quasi-isometries, with respect to composition and quasi-inverses. 

\begin{theorem}[{\cite[proposition~3.6]{GT22}}]\label{thm:stabilitypropertiesofscalingfactors}
Let $X,Y,Z$ be three bounded degree graphs, and let $f,h\colon X\longrightarrow Y$, $g\colon Y\longrightarrow Z$ be three quasi-isometries. 

\begin{enumerate}[label=(\roman*)]
    \item If $f$ and $h$ are at bounded distance and if $f$ is quasi-$k$-to-one, then $h$ is quasi-$k$-to-one. 
    \item If $f$ is quasi-$k$-to-one and $g$ is quasi-$k'$-to-one, then $g\circ f$ is quasi-$kk'$-to-one.
    \item If $f$ is quasi-$k$-to-one, then any of its quasi-inverses is quasi-$\frac{1}{k}$-to-one.
\end{enumerate}
\end{theorem}

From this theorem and the uniqueness of the scaling factor, it follows that when $X$ is amenable, there is a well-defined group morphism 
\begin{equation*}
    \text{Sc}\colon \text{QI}_{\text{sc}}(X)\longrightarrow (\R_{>0},\cdot) 
\end{equation*}
where $\text{QI}_{\text{sc}}(X) \defeq \lbrace\text{measure-scaling quasi-isometries }X\longrightarrow X\rbrace/\text{bounded distance}$. The image of this morphism, that we denote $\text{Sc}(X)$, is a subgroup of $\R_{>0}$, that we call the \textit{scaling group} of $X$. 

\smallskip

We will also require in our computations the following characterisation of rational scaling factors.

\begin{theorem}[{\cite[Theorem~4.2]{GT22}}]\label{thm:scalingfactorsandpartitions}
Let $m,n\ge 1$ be two integers and let $f\colon X\longrightarrow Y$ be a quasi-isometry between two bounded degree graphs. The following claims are equivalent:
\begin{enumerate}[label=(\roman*)]
    \item $f\colon X\longrightarrow Y$ is quasi-$\frac{m}{n}$-to-one.
    \item There exist a partition $\mathcal{P}_{X}$ (resp. $\mathcal{P}_{Y}$) of $X$ (resp. of $Y$) with uniformly bounded pieces of size $m$ (resp. $n$) and a bijection $\psi\colon \mathcal{P}_{X}\longrightarrow \mathcal{P}_{Y}$ such that $f$ is at bounded distance from a map $g\colon X\longrightarrow Y$ satisfying $g(P)\subset \psi(P)$ for every $P\in\mathcal{P}_{X}$.
\end{enumerate}
\end{theorem}

\section{Aptolic quasi-isometries}\label{section3}

\subsection{Wreath products} Given two groups $N$ and $G$, their \textit{wreath product} is the group 
\begin{equation*}
    N\wr G \defeq \left(\bigoplus_{G}N\right) \rtimes G
\end{equation*}
where $G$ acts on the direct sum by permuting the coordinates through its initial action on itself by left-multiplication. When $N$ and $G$ are both finitely generated, so is $N\wr G$, and if $N$ is generated by a finite set $T\subset N$ and $G$ by a finite set $S\subset G$, $N\wr G$ is generated by 
\begin{equation*}
    U\defeq \lbrace (\delta_{n}, 1_{G}):n\in T\rbrace\cup\lbrace (\mathbf{1}, s) : s\in S\rbrace
\end{equation*}
where, for any $n\in N$, $\delta_{n}\colon G\longrightarrow N$ is the map given by $\delta_{n}(1_{G})=n$ and $\delta_{n}(g)=1_{N}$ for $g\neq 1_{G}$, and $\mathbf{1}\equiv 1_{N}$ on $G$. Thus, in $\text{Cay}(N\wr G, U)$, two vertices $(c,p)$, $(d,q)$ are adjacent:
\begin{itemize}
    \item either if $c=d$ and $p$ and $q$ are adjacent in $\text{Cay}(G,S)$;
    \item or if $p=q$ and $c,d$ only differ on this vertex, and $c(p)$ and $d(p)$ are adjacent in $\text{Cay}(N,T)$.
\end{itemize}

The classical interpretation of such generating sets is as follows. First, think of an element $c\in \bigoplus_{G}N$ as a \textit{colouring} of the vertices of $\text{Cay}(G,S)$ with colors coming from $N$, with only finitely many vertices having a non-trivial color (and these vertices form the \textit{support} of $c$, denoted $\text{supp}(c)$). Second, think of an element $(c,p)\in N\wr G$ as a pair made of a colouring $c\in \bigoplus_{G}N$ together with an arrow pointing at some vertex $p\in G$. Then, there are two possible moves to go from $(c,p)$ to a neighbouring vertex in $\text{Cay}(N\wr G, U)$:
\begin{itemize}
    \item either we only move the arrow to a neighbouring vertex in $G$, and the colouring stays the same;
    \item or the arrow stays on the vertex where it stands, but changes the color of this vertex, and replaces it with an adjacent color in $\text{Cay}(N,T)$.
\end{itemize}

Hence, to go from $(c,p)$ to another vertex $(d,q)$ in $N\wr G$, the arrow has to move in $G$ from $p$ to $q$ by visiting all vertices of $G$ where $c$ and $d$ differ. At each of these vertices $t\in G$, the color is changed, and goes from $c(t)\in N$ to $d(t)\in N$. 

\smallskip

In the sequel, to shorten notations, we will rather denote the direct sum $\bigoplus_{G}N$ as $N^{(G)}$, and we write its composition law multiplicatively. Moreover, given a subset $A\subset G$, we denote $\mathcal{L}(A)$ the subgroup of colourings supported in $A$. 

\subsection{Aptolicity}

In this part, we recall results from~\cite{GT24b} and~\cite{BGT24} about the generic form of quasi-isometries between wreath products. This strong form of rigidity is called \textit{aptolicity}:

\begin{definition}\label{def:aptolicmap}
Let $M,N,G,H$ be finitely generated groups. A quasi-isometry $q\colon N\wr G\longrightarrow M\wr H$ is \textit{of aptolic form} if there exist two maps $\alpha \colon N^{(G)} \longrightarrow M^{(H)}$ and $\beta\colon G\longrightarrow H$ such that 
\begin{equation*}
    q(c,p)=(\alpha(c),\beta(p))
\end{equation*}
for any $(c,p)\in N\wr G$. A quasi-isometry $q\colon N\wr G\longrightarrow M\wr H$ is \textit{aptolic} if it is of aptolic form and has a quasi-inverse of aptolic form. 
\end{definition}

Roughly speaking, a quasi-isometry is aptolic if it preserves the lamplighter structure. In~\cite{GT24b}, Genevois and Tessera proved that when $N$ and $M$ are finite and $G$ and $H$ are both finitely presented and one-ended, then any quasi-isometry $N\wr G\longrightarrow M\wr H$ is at bounded distance from an aptolic quasi-isometry. Later, with Bensaid, they proved that the same phenomenon occurs when $N$ and $M$ are infinite and of subexponential growth, provided that the base groups $G$ and $H$ belong to $\mathcal{M}_{\text{exp}}$: 

\begin{theorem}[{\cite[Theorem~6.30]{BGT24}}]\label{thm:finitedistancefromaptolicmaps}
Let $A_{1},A_{2}$ be two finitely generated groups of subexponential growth, and let $B_{1},B_{2}$ be two finitely presented groups from $\mathcal{M}_{\text{exp}}$. Then any quasi-isometry $A_{1}\wr B_{1}\longrightarrow A_{2}\wr B_{2}$ is within bounded distance from an aptolic quasi-isometry. 
\end{theorem}

More precisely, it is proved in~\cite{BGT24} that any quasi-isometry $A_{1}\wr B_{1}\longrightarrow A_{2}\wr B_{2}$, with the above assumptions, has the property of being \textit{leaf-preserving}, i.e. it sends any $B_{1}-$coset of $A_{1}\wr B_{1}$ at a uniform bounded distance from a $B_{2}-$coset of $A_{2}\wr B_{2}$ and has a quasi-inverse that does the same with roles of $B_{1}$ and $B_{2}$ reversed. This fact is deduced from a powerful embedding theorem proved in~\cite[Theorem~6.1]{BGT24}. The fact that leaf-preservingness implies aptolicity (up to finite distance) is mentioned 
in~\cite[Theorem~6.28]{BGT24} and actually follows from another observation of Genevois and Tessera, recorded in~\cite[Corollary~6.11 and Lemma~6.12]{GT24a}.

\smallskip

The following characterisation of aptolic quasi-isometries will be crucial for our future goals. 

\begin{proposition}\label{prop:characterisationofaptolicity}
Let $M,N,G,H$ be finitely generated groups, and consider two maps $\alpha\colon N^{(G)}\longrightarrow M^{(H)}$ and $\beta\colon G\longrightarrow H$. Then the map 
\begin{equation*}
    \varphi\colon N\wr G\longrightarrow M\wr H, (c,p)\longmapsto (\alpha(c), \beta(p))
\end{equation*}
is an aptolic quasi-isometry if and only if the following conditions hold:
\begin{enumerate}[label=(\roman*)]
    \item $\alpha\colon N^{(G)}\longrightarrow M^{(H)}$ is a bijection;
    \item $\beta\colon G\longrightarrow H$ is a quasi-isometry;
    \item There exists $Q\ge 0$ such that, for any colourings $c_{1},c_{2}\in N^{(G)}$, the Hausdorff distance between $\text{supp}(\alpha(c_{1})^{-1}\alpha(c_{2}))$ and $\beta(\text{supp}(c_{1}^{-1}c_{2}))$ is at most $Q$;
    \item There exists $L\ge 0$ such that, for any colourings $c_{1},c_{2}\in N^{(G)}$ that differ on a single point $p\in G$ and such that $c_{1}(p),c_{2}(p)$ are adjacent in $N$, $d_{M}(\alpha(c_{1})(t), \alpha(c_{2})(t)) \le L$ for all $t\in H$;
    \item There exists $L'\ge 0$ such that, for any colourings $c_{1},c_{2}\in M^{(H)}$ that differ on a single point $p\in H$ and such that $c_{1}(p),c_{2}(p)$ are adjacent in $M$, $d_{N}(\alpha^{-1}(c_{1})(t), \alpha^{-1}(c_{2})(t)) \le L'$ for all $t\in G$.
\end{enumerate}
In particular, every aptolic quasi-inverse of $\varphi$ is of the form 
\begin{equation*}
    (c,p)\longmapsto (\alpha^{-1}(c), \overline{\beta}(p)),\; (c,p)\in M\wr H
\end{equation*}
where $\overline{\beta}\colon H\longrightarrow G$ is a quasi-inverse of $\beta\colon G\longrightarrow H$. 
\end{proposition}

\begin{proof}
Suppose first that $\varphi\colon N\wr G\longrightarrow M\wr H$ is a $(C,K)-$quasi-isometry.
The proof of \textit{(i)}, \textit{(ii)}, \textit{(iii)} and of the last claim of the statement follow the same lines as in~\cite[Proposition~3.1]{GT24b}.  

\smallskip

Let us then turn to \textit{(iv)}. Let $c_{1},c_{2}\in N^{(G)}$ be two colourings of $G$ that differ only on $p\in G$, and take adjacent values on that point. This means that $(c_{1},p)$ and $(c_{2},p)$ are neighbours in $N\wr G$, and thus 
\begin{align*}
    d_{M\wr H}((\alpha(c_{1}),\beta(p)), (\alpha(c_{2}),\beta(p)))&=d_{M\wr H}(\varphi(c_{1},p), \varphi(c_{2},p)) \\
    &\le C\cdot d_{N\wr G}((c_{1},p), (c_{2},p))+K \\
    &=C+K.
\end{align*}
This distance being less than $C+K$ implies that $\alpha(c_{1})$ and $\alpha(c_{2})$ can only differ on points at distance $\le C+K$ from $\beta(p)$ and, at each point $t\in H$ where they effectively differ, one can go in $M$ from $\alpha(c_{1})(t)$ to $\alpha(c_{2})(t)$ in at most $C+K$ steps. Thus $d_{M}(\alpha(c_{1})(t), \alpha(c_{2})(t))\le C+K$ for any $t\in H$. 

\smallskip

The proof of \textit{(v)} is analogous, using a quasi-inverse $\overline{\varphi}$ of $\varphi$ that takes the form 
\begin{equation*}
    (c,p)\longmapsto (\alpha^{-1}(c), \overline{\beta}(p)),\; (c,p)\in M\wr H
\end{equation*}
where $\overline{\beta}\colon H\longrightarrow G$ is a quasi-inverse of $\beta$. 

\smallskip

Conversely, assume that $\textit{(i)}-\textit{(v)}$ hold. Fix constants $C\ge 1,K\ge 0$ such that $\beta\colon G\longrightarrow H$ and a quasi-inverse $\overline{\beta}\colon H\longrightarrow G$ are $(C,K)-$quasi-isometries, and such that $\overline{\beta}\circ\beta$, $\beta\circ\overline{\beta}$ are at distance $\le K$ from $\text{Id}_{G}$, $\text{Id}_{H}$ respectively. 

\smallskip

We start by showing that $\varphi$ is Lipschitz. Let $a=(c_{1},p_{1})$, $b=(c_{2},p_{2})$ be two adjacent vertices in $N\wr G$. We distinguish two cases:
\begin{itemize}
    \item First, say that $c_{1}=c_{2}$ and that $p_{1}$ and $p_{2}$ are neighbours in $G$. Then we have 
    \begin{align*}
        d_{M\wr H}(\varphi(a),\varphi(b))&=d_{M\wr H}\left((\alpha(c_{1}),\beta(p_{1})), (\alpha(c_{1}),\beta(p_{2}))\right) \\
        &=d_{H}(\beta(p_{1}), \beta(p_{2})) \\
        &\le C\cdot d_{G}(p_{1},p_{2})+K \\
        &=C+K.
    \end{align*}
    \item Now, assume that $p_{1}=p_{2}$ and that $c_{1},c_{2}$ differ only on $p_{1}$, and that $c_{1}(p_{1})$ and $c_{2}(p_{1})$ are adjacent in $N$. From \textit{(iii)}, we know that $\text{supp}(\alpha(c_{1})^{-1}\alpha(c_{2}))\subset B_{H}(\beta(p_{1}),Q)$ and, from \textit{(iv)}, we know that for any $t\in \text{supp}(\alpha(c_{1})^{-1}\alpha(c_{2}))$, there is a path of length $\le L$ in $M$ from $\alpha(c_{1})(t)$ to $\alpha(c_{1})(t)$. This yields the estimate
    \begin{align*}
        d_{M\wr H}(\varphi(a),\varphi(b))&=d_{M\wr H}\left((\alpha(c_{1}),\beta(p_{1})), (\alpha(c_{2}),\beta(p_{1}))\right) \\
        &\le L\cdot \left|B_{H}(\beta(p_{1}),Q)\right| \\
        &\le L\cdot D^{Q}
    \end{align*}
    where $D\ge 3$ is a fixed integer larger than the degree of a vertex in $\text{Cay}(H,T)$.
\end{itemize}
In all cases, we conclude that 
\begin{equation*}
    d_{M\wr H}(\varphi(a),\varphi(b)) \le \max(C+K, L\cdot D^{Q})
\end{equation*}
so that $\varphi$ is \mbox{$\max(C+K, L\cdot D^{Q})-$}Lipschitz according to Lemma~\ref{lm:lipschitzmap}. 

\smallskip

Now, consider the map 
\begin{align*}
    \psi \colon M\wr H &\longrightarrow N\wr G \\
    (c,p)&\longmapsto (\alpha^{-1}(c), \overline{\beta}(p)). 
\end{align*}
For any $(c,p)\in N\wr G$, one computes that 
\begin{align*}
    d_{N\wr G}(\psi\circ\varphi(c,p), (c,p))&=d_{N\wr G}((c, \overline{\beta}(\beta(p))), (c,p)) \\
    &=d_{G}(\overline{\beta}\circ\beta(p), p) \\
    &\le K
\end{align*}
so $\psi\circ\varphi$ is at distance $\le K$ from $\text{Id}_{N\wr G}$, and similarly, using that $\beta\circ\overline{\beta}$ is at distance $\le K$ from $\text{Id}_{H}$, one shows that $\varphi\circ \psi$ is at distance $\le K$ from $\text{Id}_{M\wr H}$. 

\smallskip

Thus it only remains to prove that $\psi$ is Lipschitz. For this, note that:

\begin{claim}\label{claim3.4}
There exists $Q'\ge 0$ such that, for any $c_{1},c_{2}\in M^{(H)}$, the Hausdorff distance between $\text{supp}(\alpha^{-1}(c_{1})^{-1}\alpha^{-1}(c_{2}))$ and $\overline{\beta}(\text{supp}(c_{1}^{-1}c_{2}))$ is at most $Q'$.
\end{claim}

{
\renewcommand{\proofname}{Proof of Claim \ref{claim3.4}.}
\renewcommand{\qedsymbol}{$\blacksquare$}
\begin{proof}
Fix $c_{1},c_{2}\in M^{(H)}$. Applying \textit{(iii)} to $\alpha^{-1}(c_{1})$ and $\alpha^{-1}(c_{2})$, we have  
\begin{equation*}
    d_{\text{Haus}}\left(\text{supp}(c_{1}^{-1}c_{2}), \beta\left(\text{supp}(\alpha^{-1}(c_{1})^{-1}\alpha(c_{2}))\right)\right) \le Q
\end{equation*}
and applying $\overline{\beta}$, it follows that 
\begin{equation*}
    d_{\text{Haus}}\left(\overline{\beta}(\text{supp}(c_{1}^{-1}c_{2})), \overline{\beta}\left(\beta\left(\text{supp}(\alpha^{-1}(c_{1})^{-1}\alpha(c_{2}))\right)\right)\right) \le C\cdot Q+K.
\end{equation*}
As the Hausdorff distance between $\overline{\beta}\left(\beta\left(\text{supp}(\alpha^{-1}(c_{1})^{-1}\alpha(c_{2}))\right)\right)$ and $\text{supp}(\alpha^{-1}(c_{1})^{-1}\alpha(c_{2}))$ is bounded by $K$, we get the claim with $Q'\defeq C\cdot Q+2K$. 
\end{proof}
\renewcommand{\qedsymbol}{$\square$}
}
Using Claim~\ref{claim3.4} and \textit{(v)}, computations as the ones above show that $\psi$ sends any two adjacent vertices of $M\wr H$ to vertices that are at distance at most $\max(C+K, L'\cdot D'^{Q'})$ in $N\wr G$, where $D'\ge 3$  is a fixed integer larger than the degree of a vertex in $\text{Cay}(G,S)$. Invoking once again Lemma~\ref{lm:lipschitzmap}, we conclude that $\psi$ is $\max(C+K, L'\cdot D'^{Q'})-$Lipschitz. Finally, we get that 
\begin{align*}
d_{M\wr H}(\varphi(a),\varphi(b))&\ge \frac{1}{\max(C+K, L'\cdot D'^{Q'})}d_{N\wr G}\left(\psi(\varphi(a)), \psi(\varphi(b))\right) \\
&\ge \frac{1}{\max(C+K, L'\cdot D'^{Q'})}(d_{N\wr G}(a,b)-2K) 
\end{align*}
for any $a,b\in N\wr G$, which concludes the proof that $\varphi\colon N\wr G\longrightarrow M\wr H$ is a quasi-isometry, with $\psi\colon M\wr H\longrightarrow N\wr G$ as a quasi-inverse. 
\end{proof}

\smallskip

Here is then the key statement that will provide us with the rigidity we are looking for. 

\begin{proposition}\label{prop:finiteunionofcosets}
Let $N,M,G,H$ be finitely generated groups, and let $\varphi\colon N\wr G\longrightarrow M\wr H$ be an aptolic quasi-isometry, i.e. there are two maps $\alpha\colon N^{(G)}\longrightarrow M^{(H)}$ and $\beta\colon G\longrightarrow H$ such that $\varphi(c,p)=(\alpha(c),\beta(p))$ for any $(c,p)\in N\wr G$. 

For any quasi-inverse $\overline{\beta}$ of $\beta$, there exists a constant $Q\ge 0$ such that, for any finite subset $A\subset G$ and any $Q'\ge Q$, $\alpha^{-1}(\mathcal{L}(\beta(A)^{+Q'}))$ is a union of cosets of $\mathcal{L}(A)$; and conversely, for any finite subset $B\subset H$ and any $Q'\ge Q$, $\alpha(\mathcal{L}(\overline{\beta}(B)^{+Q'}))$ is a union of cosets of $\mathcal{L}(B)$.
\end{proposition}

Recall that, for a subset $A\subset G$, $\mathcal{L}(A)$ stands for the subgroup of $N^{(G)}$ of colourings supported on $A$. 

The proof of this observation requires two intermediate lemmas.

\begin{lemma}\label{lm:finiteunionofcosets}
Let $N,M,G,H$ be finitely generated groups, and consider two maps $\alpha\colon N^{(G)}\longrightarrow M^{(H)}$ and $\beta\colon G\longrightarrow H$. Suppose that there exists $Q\ge 0$ such that, for any $c_{1},c_{2}\in N^{(G)}$ with $\text{supp}(c_{1}^{-1}c_{2})\subset \lbrace p\rbrace$ for some $p\in G$, we have $\text{supp}(\alpha(c_{1})^{-1}\alpha(c_{2})) \subset B_{H}(\beta(p),Q)$. 

Then, for any finite subset $A\subset X$ and any colouring $c\in N^{(G)}$, we have 
\begin{equation*}
    \alpha(c\mathcal{L}(A))\subset \alpha(c)\mathcal{L}(\beta(A)^{+Q}).
\end{equation*}
\end{lemma}

\begin{proof}
Let $c\in N^{(G)}$. We prove the claim by induction on $|A|$. Suppose to start that $|A|=1$ and $c\in N^{(G)}$, and let $c'\in c\mathcal{L}(A)$. Then $c$ and $c'$ differ on at most one point $p$ of $G$, so our assumption implies that $\alpha(c)$ and $\alpha(c')$ can only differ on points from $B_{H}(\beta(p),Q)=\beta(\lbrace p\rbrace)^{+Q}=\beta(A)^{+Q}$. This means that $\alpha(c')\in \alpha(c)\mathcal{L}(\beta(A)^{+Q})$, so the claim holds for subsets reduced to a point. 

\smallskip

Suppose now that it holds for any colouring and for subsets of cardinality $k\ge 1$, and let $A\subset X$ be a set of cardinality $k+1$. Let $c\in N^{(G)}$ and let $c'\in c\mathcal{L}(A)$. Hence $c^{-1}c'\in\mathcal{L}(A)$, and we can write $c^{-1}c'=c''d$ for some colourings $c''\in \mathcal{L}(A\setminus\lbrace a\rbrace)$, $d\in\mathcal{L}(\lbrace a\rbrace)$, where $a\in A$ is an arbitrary point. Thus $c'=(cc'')d$, and it follows that 
\begin{equation*}
    \alpha(c')\in \alpha(cc'')\mathcal{L}(\beta(\lbrace a\rbrace)^{+Q}).
\end{equation*}
On the other hand, the inductive assumption applied to $A\setminus\lbrace a\rbrace$ shows that \mbox{$\alpha(cc'')\in\alpha(c)\mathcal{L}(\beta(A\setminus\lbrace a\rbrace)^{+Q})$}, whence 
\begin{equation*}
    \alpha(c')\in \alpha(cc'')\mathcal{L}(\beta(\lbrace a\rbrace)^{+Q}) \subset \alpha(c)\mathcal{L}(\beta(A\setminus\lbrace a\rbrace)^{+Q})\mathcal{L}(\beta(\lbrace a\rbrace)^{+Q})=\alpha(c)\mathcal{L}(\beta(A)^{+Q}).
\end{equation*}
Thus we conclude that $\alpha(c\mathcal{L}(A))\subset \alpha(c)\mathcal{L}(\beta(A)^{+Q})$, and the induction is complete. 
\end{proof}

{
\renewcommand{\proofname}{Proof of Proposition~\ref{prop:finiteunionofcosets}.}

\begin{proof} 
By Proposition~\ref{prop:characterisationofaptolicity}, we know that there exists $Q\ge 0$ such that the Hausdorff distance between $\beta(\text{supp}(c_{1}^{-1}c_{2}))$ and $\text{supp}(\alpha(c_{1})^{-1}\alpha(c_{2}))$ is at most $Q$, for any $c_{1},c_{2}\in N^{(G)}$. In particular, the assumption of Lemma~\ref{lm:finiteunionofcosets} is satisfied. 

\smallskip

Fix now any finite subset $A\subset G$ and a number $Q'\ge Q$. Let $c'\in \alpha^{-1}(\mathcal{L}(\beta(A)^{+Q'}))$, and let $d\in \mathcal{L}(A)$. Then, applying Lemma~\ref{lm:finiteunionofcosets}, one has 
\begin{equation*}
    \alpha(c'd) \in \alpha(c')\mathcal{L}(\beta(A)^{+Q}) \subset \alpha(c')\mathcal{L}(\beta(A)^{+Q'})
\end{equation*}
and since $\alpha(c')\in \mathcal{L}(\beta(A)^{+Q'})$, we deduce that $\alpha(c'd)\in \mathcal{L}(\beta(A)^{+Q'})$, i.e. 
\begin{equation*}
    c'd\in \alpha^{-1}(\mathcal{L}(\beta(A)^{+Q'})).
\end{equation*}
Thus $\alpha^{-1}(\mathcal{L}(\beta(A)^{+Q'}))$ is stable under multiplication by elements of $\mathcal{L}(A)$, which implies that it is a union of cosets of $\mathcal{L}(A)$.
\end{proof}}

Let us conclude this section with a sufficient condition for aptolic quasi-isometries to be scaling.

\begin{lemma}\label{lm:scalingfactorsforwreathproducts}
Let $N,M,G,H$ be finitely generated groups, and consider two maps $\alpha\colon N^{(G)}\longrightarrow M^{(H)}$ and $\beta\colon G\longrightarrow H$ such that the map $\varphi\colon N\wr G\longrightarrow M\wr H$, $\varphi(c,p)=(\alpha(c),\beta(p))$ is a quasi-isometry. If $\beta$ is quasi-$k$-to-one, then $\varphi$ is quasi-$k$-to-one. 
\end{lemma}

\begin{proof}
The proof of~\cite[Lemma~3.9]{GT24b} applies word for word. 
\end{proof}
\section{Proof of Theorem~\ref{thm:maintheorem}}\label{section4}

We can now combine all these intermediate observations to prove our main result. Fix then \mbox{$\varphi\colon N\wr G\longrightarrow M\wr H$} an arbitrary quasi-isometry, where $N$ and $M$ have polynomial growth of degrees $n$ and $m$ respectively. Our goal is to prove that $\varphi$ is quasi-$\frac{m}{n}$-to-one.

\smallskip

Let $C\ge 1$, $K\ge 0$ be such that $\varphi$ and a quasi-inverse $\overline{\varphi}\colon M\wr H\longrightarrow N\wr G$ are $(C,K)-$quasi-isometries, and such that  $\overline{\varphi}\circ\varphi$, $\varphi\circ\overline{\varphi}$ are within distance $K$ from $\text{Id}_{N\wr G}$, $\text{Id}_{M\wr H}$ respectively. By assumption, up to a bounded distance, $\varphi$ can be chosen aptolic, i.e. there are two maps $\alpha\colon N^{(G)}\longrightarrow M^{(H)}$, $\beta\colon G\longrightarrow H$ such that
\begin{equation*}
    \varphi(c,p)=(\alpha(c),\beta(p))
\end{equation*}
for any $(c,p)\in N\wr G$. It follows from Proposition~\ref{prop:characterisationofaptolicity} that $\alpha$ is a bijection, that $\beta$ is a $(C,K)-$quasi-isometry, and that $\overline{\varphi}$ takes the form
\begin{equation*}
    \overline{\varphi}(c,p)=(\alpha^{-1}(c), \overline{\beta}(p))
\end{equation*}
for any $(c,p)\in M\wr H$, where $\overline{\beta}\colon H\longrightarrow G$ is a quasi-inverse of $\beta$ (with same parameters $C$ and $K$). Lastly, by increasing $K$ if needed, we may assume that it is larger than the constant $Q\ge 0$ given by Proposition~\ref{prop:characterisationofaptolicity}\textit{(iii)}. 

\begin{claim}\label{claim4.1}
The equality $\left|A\right|=\frac{m}{n}\left|\beta(A)^{+K}\right|$ holds for any finite subset $A\subset G$. 
\end{claim}

{
\renewcommand{\proofname}{Proof of Claim~\ref{claim4.1}.}
\renewcommand{\qedsymbol}{$\blacksquare$}
\begin{proof}
Let $A\subset G$ be finite. Notice that, by Proposition~\ref{prop:finiteunionofcosets}, $\alpha^{-1}(\mathcal{L}(\beta(A)^{+2K}))$ is a union of cosets of $\mathcal{L}(A)$, say $\alpha^{-1}(\mathcal{L}(\beta(A)^{+2K}))=\bigsqcup_{i=1}^{k}(d_{i}\mathcal{L}(A))$ for some $k\ge 1$ and $d_{1},\dots, d_{k}\in N^{(G)}$. Consider then the subset 
\begin{equation*}
    S \defeq \left\lbrace (d,q) \in M\wr H : d\in \mathcal{L}(\beta(A)^{+K}), q\in \beta(A)^{+K}\right\rbrace. 
\end{equation*}
Notice that $S^{+K}\subset \left\lbrace (d,q)\in M\wr H : d\in \mathcal{L}(\beta(A)^{+2K}), q\in \beta(A)^{+2K}\right\rbrace$, so that 
\begin{align*}
    \varphi^{-1}(S^{+K}) &\subset \varphi^{-1}\left(\left\lbrace (d,q)\in M\wr H : d\in \mathcal{L}(\beta(A)^{+2K}), q\in \beta(A)^{+2K}\right\rbrace\right) \\
    &=\left\lbrace (c,p) \in N\wr G : c\in \alpha^{-1}(\mathcal{L}(\beta(A)^{+2K})), p\in \beta^{-1}(\beta(A)^{+2K})\right\rbrace \\
    &=\left\lbrace (c,p) \in N\wr G : c\in \bigsqcup_{i=1}^{k}d_{i}\mathcal{L}(A), p\in \beta^{-1}(\beta(A)^{+2K})\right\rbrace
\end{align*}
and the latter has growth degree $n|A|$. Hence $\varphi^{-1}(S^{+K})$ has growth degree at most $n|A|$. On the other hand, $\varphi^{-1}(S^{+K})$ coarsely coincides with $\overline{\varphi}(S)$ by Lemma~\ref{lm:preimagesandquasiinverses}, so their growth degrees coincide, and as moreover the growth degree is preserved by quasi-isometries, it follows that $\varphi^{-1}(S^{+K})$ has the same growth degree as $S$, which is $m\left|\beta(A)^{+K}\right|$. We thus get the inequality
\begin{equation}\label{eq4.1}
    m\left|\beta(A)^{+K}\right| \le n\left|A\right|
\end{equation}
for any finite subset $A\subset G$. The same reasoning with $\overline{\varphi}$ shows that 
\begin{equation}\label{eq4.2}
    n\left|\overline{\beta}(B)^{+K}\right| \le m\left|B\right|
\end{equation}
for any finite subset $B\subset H$. Thus, fixing any finite subset $A\subset G$ and applying (\ref{eq4.2}) with $B=\beta(A)^{+K}$ and then (\ref{eq4.1}), it follows that 
\begin{equation*}
    n\left|\overline{\beta}(\beta(A)^{+K})^{+K}\right| \le m\left|\beta(A)^{+K}\right| \le n\left|A\right|
\end{equation*}
Since $A\subset \overline{\beta}(\beta(A)^{+K})^{+K}$ (any $x\in A$ is within distance $K$ from $\overline{\beta}(\beta(x))$ with $\beta(x)\in\beta(A)^{+K}$), we get that 
\begin{equation*}
    n\left|A\right| \le n\left|\overline{\beta}(\beta(A)^{+K})^{+K}\right| \le m\left|\beta(A)^{+K}\right| \le n\left|A\right|
\end{equation*}
and finally $m\left|\beta(A)^{+K}\right|=n\left|A\right|$, i.e. $\left|A\right|=\frac{m}{n}\left|\beta(A)^{+K}\right|$ for any finite subset $A\subset G$. This proves Claim~\ref{claim4.1}.
\end{proof}
\renewcommand{\qedsymbol}{$\square$}
}

\begin{claim}\label{claim4.2}
The quasi-isometry $\beta\colon G\longrightarrow H$ is quasi-$\frac{m}{n}$-to-one.
\end{claim}

{
\renewcommand{\proofname}{Proof of Claim~\ref{claim4.2}.}
\renewcommand{\qedsymbol}{$\blacksquare$}
\begin{proof}
Fix any finite subset $A\subset H$. Using Claim~\ref{claim4.1}, we estimate 
\begin{align*}
    \left|\frac{m}{n}|A|-|\beta^{-1}(A)|\right| &= \left|\frac{m}{n}|A|-\frac{m}{n}|\beta(\beta^{-1}(A))^{+K}|\right| \\
    &\le \left|\frac{m}{n}|A|-\frac{m}{n}|A^{+K}|\right|+\left|\frac{m}{n}|A^{+K}|-\frac{m}{n}|\beta(\beta^{-1}(A))^{+K}|\right| \\
    &\le \frac{m}{n}\left|A^{+K}\setminus A\right|+\frac{m}{n}\left|A^{+K}\setminus \beta(\beta^{-1}(A))^{+K}\right|.
\end{align*}
Now, we claim that $A^{+K}\setminus \beta(\beta^{-1}(A))^{+K} \subset (\partial_{H}A)^{+(2K-1)}$. Indeed, let $y\in A^{+K}\setminus \beta(\beta^{-1}(A))^{+K}$. We distinguish two cases. If $y\notin A$, then $y\in A^{+K}\setminus A \subset (\partial_{H}A)^{+(K-1)}\subset (\partial_{H}A)^{+(2K-1)}$ as wanted.

\smallskip

We may therefore assume that $y\in A$. We know that there is $x\in G$ with $d_{H}(y,\beta(x))\le K$. In particular, $\beta(x)\in A^{+K}$. On the other hand, since $y\notin  \beta(\beta^{-1}(A))^{+K}$, it follows that $x\notin \beta^{-1}(A)$, i.e. $\beta(x)\notin A$. Thus $\beta(x)\in A^{+K}\setminus A \subset (\partial_{H}A)^{+(K-1)}$. As $y$ is within distance $K$ from $\beta(x)$, it follows that $y$ is within distance $2K-1$ from $\partial_{H}A$, as wanted. 

\smallskip

In any case, we get $y\in (\partial_{H}A)^{+(2K-1)}$, which proves that $A^{+K}\setminus \beta(\beta^{-1}(A))^{+K} \subset (\partial_{H}A)^{+(2K-1)}$. 

\smallskip

Therefore, it follows from Lemma~\ref{lm:cardinalityofneighborhood} and Lemma~\ref{lm:cardinalityofneighborhood2} that there exist constants $L>0$, $R>0$ (depending only on $K$ and $H$) such that 
\begin{equation*}
    \left|A^{+K}\setminus A\right| \le R\left|\partial_{H}A\right|
\end{equation*}
and 
\begin{equation*}
    \left|A^{+K}\setminus \beta(\beta^{-1}(A))^{+K}\right| \le \left|(\partial_{H}A)^{+(2K-1)}\right| \le L\left|\partial_{H}A\right|.
\end{equation*}
Hence we finish our estimation above as
\begin{equation*}
    \left|\frac{m}{n}|A|-|\beta^{-1}(A)|\right| \le \frac{m}{n}(L+R)|\partial_{H}A|
\end{equation*}
which proves Claim~\ref{claim4.2}.
\end{proof}
\renewcommand{\qedsymbol}{$\square$}
}

Now, the fact that $\varphi$ is quasi-$\frac{m}{n}$-to-one directly follows from Claim~\ref{claim4.2} and Lemma~\ref{lm:scalingfactorsforwreathproducts}, which proves Theorem~\ref{thm:maintheorem}.
\section{Applications}\label{section5}

We can now focus on the consequences mentioned in the introduction.

{
\renewcommand{\proofname}{Proof of Theorem~\ref{thm:maintheoremfortheclassMexp}}
\begin{proof}
By Theorem~\ref{thm:finitedistancefromaptolicmaps}, any quasi-isometry $N\wr G\longrightarrow M\wr H$ is at a bounded distance from an aptolic quasi-isometry. Thus Theorem~\ref{thm:maintheorem} applies and give the conclusion.
\end{proof}}

{
\renewcommand{\proofname}{Proof of Corollary~\ref{cor:BLE}}
\begin{proof}
Assume that $M\wr G$ and $N\wr H$ are biLipschitz equivalent. Such a biLipschitz equivalence is quasi-one-to-one, and up to a finite distance change, it can be taken aptolic. This new quasi-isometry, that we denote 
\begin{align*}
    \varphi\colon N\wr G &\longrightarrow M\wr H \\
    (c,p)&\longmapsto (\alpha(c),\beta(p))
\end{align*}
is still quasi-one-to-one by Theorem~\ref{thm:stabilitypropertiesofscalingfactors}(\textit{i}), and in addition it must be quasi-$\frac{m}{n}$-to-one by Theorem~\ref{thm:maintheoremfortheclassMexp}. From Lemma~\ref{lm:uniquenessofscalingfactors}, it follows that $m=n$, and we also deduce that $\beta\colon G\longrightarrow H$ is quasi-one-to-one. From Theorem~\ref{thm:whytetheorem}, $\beta$ lies at finite distance from a bijection, which is the desired biLipschitz equivalence $G\longrightarrow H$. 
\end{proof}}

The proofs of Theorem~\ref{thm:classificationforvirtuallyabeliangroups} and Corollary~\ref{cor:classificationforvirtuallyabeliangroups2} require, in the flexibility part, to be able to construct quasi-isometries between wreath products in general situations. This is the goal of the next criterion. 

\begin{proposition}\label{prop:constructionsofQI}
Let $N,M,G,H$ be finitely generated groups. Let $n,m\ge 2$. If there exist a biLipschitz equivalence $N^{m}\longrightarrow M^{n}$ and a quasi-$\frac{m}{n}$-to-one quasi-isometry $G\longrightarrow H$, then there exists an aptolic quasi-isometry
\begin{equation*}
    N\wr G\longrightarrow M\wr H. 
\end{equation*}
\end{proposition}

\begin{proof}
Let $\sigma\colon N^{m}\longrightarrow M^{n}$ be a $C-$biLipschitz equivalence, and let $\beta\colon G\longrightarrow H$ be quasi-$\frac{m}{n}$-to-one. From Theorem~\ref{thm:scalingfactorsandpartitions}, there exist a partition $\mathcal{P}$ (resp. $\mathcal{Q}$) of $G$ (resp. of $H$) with uniformly bounded pieces of size $m$ (resp. of size $n$), and a bijection $\psi\colon \mathcal{P}\longrightarrow \mathcal{Q}$ such that $\beta(P)\subset \psi(P)$ for all $P\in \mathcal{P}$. Up to postcomposing $\sigma$ with a translation by an element of $M^{n}$, we can assume that $\sigma(1_{N^{m}})=1_{M^{n}}$. We can now define a bijection $\alpha\colon N^{(G)}\longrightarrow M^{(H)}$ in such a way that $\alpha$ sends $\mathcal{L}(P)$ into $\mathcal{L}(\psi(P))$ through $\sigma$ for any $P\in\mathcal{P}$, and $\alpha^{-1}$ sends $\mathcal{L}(Q)$ into $\mathcal{L}(\psi^{-1}(Q))$ through $\sigma^{-1}$ for any $Q\in\mathcal{Q}$.
Set 
\begin{align*}
    f\colon N\wr G &\longrightarrow M\wr H \\
    (c,p) &\longmapsto (\alpha(c), \beta(p)).
\end{align*}
We show that $f$ is a quasi-isometry by checking the five points of Proposition~\ref{prop:characterisationofaptolicity}. Points \textit{(i)} and \textit{(ii)} are satisfied by construction. For \textit{(iv)}, fix two colourings $c_{1},c_{2}\in N^{(G)}$ that differ on a single point $p\in G$ and such that $d_{N}(c_{1}(p), c_{2}(p))=1$. Let $P\in \mathcal{P}$ be the piece containing $p$. Let $u\in N^{m}$ (resp. $v\in N^{m}$) be the vector formed by colors of $c_{1}$ (resp. of $c_{2}$) on the piece $P$. Then $d_{N^{m}}(u,v)=1$, and since $\sigma$ is $C-$Lipschitz, it follows that 
\begin{equation}\label{eq5.2}
    d_{M^{n}}(\sigma(u),\sigma(v))\le C.
\end{equation}
By construction, the components of $\sigma(u)\in M^{n}$ (resp. of $\sigma(v)\in M^{n}$) are the colors of $\alpha(c_{1})$ (resp. of $\alpha(c_{2})$) on the piece $\psi(P)$. Thus, from (\ref{eq5.2}), it follows that 
\begin{equation*}
    d_{M}(\alpha(c_{1})(t), \alpha(c_{2})(t)) \le C
\end{equation*}
for any $t\in\psi(P)$. Additionally, if $t\in H\setminus \psi(P)$, then it belongs to another piece $Q\in \mathcal{Q}$, that we may write $Q=\psi(P')$ for some $P'\neq P$. By assumption, $c_{1}$ and $c_{2}$ agree on $P'$, so $\alpha(c_{1})$ and $\alpha(c_{2})$ agree on $\psi(P')$, in particular on $t\in \psi(P')$. Thus we have proved 
\begin{equation*}
    d_{M}(\alpha(c_{1})(t), \alpha(c_{2})(t)) \le C
\end{equation*}
for any $t\in H$, which is exactly \textit{(iv)} of Proposition~\ref{prop:characterisationofaptolicity}. Point \textit{(v)} is checked in a similar manner.

\smallskip

Finally, let us focus on \textit{(iii)}. Fix two colourings $c_{1},c_{2}\in N^{(G)}$, and let $P_{1},\dots, P_{k}\in\mathcal{P}$ be the pieces of $\mathcal{P}$ containing points of $\text{supp}(c_{1}^{-1}c_{2})$. Then $\psi(P_{1}),\dots,\psi(P_{k})$ are the pieces of $\mathcal{Q}$ containing points of $\beta(\text{supp}(c_{1}^{-1}c_{2}))$. Since those pieces are uniformly bounded, we deduce that there is a constant $D\ge 0$ such that
\begin{equation*}
    d_{\text{Haus}}\left(\text{supp}(\alpha(c_{1})^{-1}\alpha(c_{2})), \psi(P_{1})\cup\dots\cup\psi(P_{k})\right) \le D.
\end{equation*}
Additionally, since $\beta(\text{supp}(c_{1}^{-1}c_{2})) \subset \psi(P_{1})\cup\dots\cup\psi(P_{k})$ and has a point in each of these pieces, there is some $D'\ge 0$ such that 
\begin{equation*}
    d_{\text{Haus}}\left(\beta(\text{supp}(c_{1}^{-1}c_{2})), \psi(P_{1})\cup\dots\cup\psi(P_{k})\right) \le D'.
\end{equation*}
We conclude that 
\begin{equation*}
    d_{\text{Haus}}\left(\beta(\text{supp}(c_{1}^{-1}c_{2})), \text{supp}(\alpha(c_{1})^{-1}\alpha(c_{2}))\right) \le D+D'
\end{equation*}
where $D$ and $D'$ are independent of $c_{1},c_{2}$. This shows \textit{(iii)} of Proposition~\ref{prop:characterisationofaptolicity} and completes the proof that $f$ is a quasi-isometry.
\end{proof}

We already know from~\cite[Fact~6.35]{BGT24} that the converse statement is not true; namely a quasi-isometry between $N\wr G$ and $M\wr H$ does not necessarily provide a quasi-isometry between a power of $N$ and a power of $M$.

\begin{remark}
Proposition~\ref{prop:constructionsofQI} can be adapted to construct quasi-isometries between permutational wreath products with infinite lamp groups, thus providing amenable analogs of the flexibility already observed in~\cite[Remark~5.2]{Dum24}. For instance, combining the above method with the one of~\cite[Proposition~5.6]{Dum24} shows that $\Z^{n}\wr_{\Z^k}\Z^d$ and $\Z^{n'}\wr_{\Z^k}\Z^d$ are quasi-isometric for any $n,n'\ge 1$ and $d\ge k\ge 1$. 
\end{remark}

We will also require the next observation, which is a variant of~\cite[Lemma~1]{Dyu00}.

\begin{lemma}\label{lm:BLEbetweenwreathproducts}
Let $N,G,H$ be finitely generated groups. If $G$ and $H$ are biLipschitz equivalent, then $N\wr G$ and $N\wr H$ are biLipschitz equivalent. 
\end{lemma}

\begin{proof}
Fix a $C-$biLipschitz equivalence $f\colon G\longrightarrow H$ with $C\ge 1$, and consider the map
\begin{align*}
    \varphi\colon N\wr G&\longrightarrow N\wr H \\
    (c,p)&\longmapsto (c\circ f^{-1}, f(p)).
\end{align*}
Let $a=(c,p)$ and $b=(d,q)$ be adjacent vertices in $N\wr G$. This means that:
\begin{itemize}
    \item either $c=d$ and $p$ and $q$ are adjacent in $G$, in which case we get
    \begin{equation*}
        d_{N\wr H}(\varphi(a), \varphi(b))=d_{N\wr H}((c\circ f^{-1}, f(p)), (c\circ f^{-1}, f(q)))=d_{H}(f(p), f(q)) \le C
    \end{equation*}
    since $f$ is $C-$Lipschitz.
    \item or $p=q$ and $c$ and $d$ only differ on $p$, with $c(p)$ and $d(p)$ being adjacent in $N$. In this case, $c\circ f^{-1}$ and $d\circ f^{-1}$ only differ on $f(p)\in H$, and take adjacent colors in $N$ on that vertex. Therefore in this case as well
    \begin{equation*}
        d_{N\wr H}(\varphi(a), \varphi(b))=d_{N\wr H}((c\circ f^{-1}, f(p)), (d\circ f^{-1}, f(p)))=1\le C.
    \end{equation*}
\end{itemize}
It thus follows from Lemma~\ref{lm:lipschitzmap} that $f$ is $C-$Lipschitz. Now, as $f$ is bijective, so is $\varphi$, and its inverse is given by 
\begin{align*}
    \psi\colon N\wr H&\longrightarrow N\wr G \\
    (c,p)&\longmapsto (c\circ f, f^{-1}(p)).
\end{align*}
A computation as the one above, using that $f^{-1}$ is $C-$Lipschitz, proves that $\psi$ is $C-$Lipschitz as well. Hence, we conclude that 
\begin{equation*}
    \frac{1}{C}\cdot d_{N\wr G}(a,b)=\frac{1}{C}\cdot d_{N\wr G}(\psi(\varphi(a)),\psi(\varphi(b))\le d_{N\wr H}(\varphi(a), \varphi(b)) \le C\cdot d_{N\wr G}(a,b)
\end{equation*}
for any $a,b\in N\wr G$, and the proof is complete.
\end{proof}

We can now deduce Theorem~\ref{thm:classificationforvirtuallyabeliangroups} from these observations and our main theorem. 

{
\renewcommand{\proofname}{Proof of Theorem~\ref{thm:classificationforvirtuallyabeliangroups}}
\begin{proof}
We start proving \textit{(i)}. First of all, if there is a quasi-isometry $\varphi\colon A_{1}\wr G\longrightarrow A_{2}\wr H$, then up to finite distance, $\varphi$ can be chosen aptolic, so we write $\varphi(c,p)=(\alpha(c),\beta(p))$ for any $(c,p)\in A_{1}\wr G$. From (the proof of) Theorem~\ref{thm:maintheorem}, we know then that $\beta\colon G\longrightarrow H$ is quasi-$\frac{d_{2}}{d_{1}}$-to-one. 

\smallskip

Conversely, assume that there exists a quasi-$\frac{d_{2}}{d_{1}}$-to-one quasi-isometry $G\longrightarrow H$. Since there exists a biLipschitz equivalence $(\Z^{d_{1}})^{d_{2}}\longrightarrow (\Z^{d_{2}})^{d_{1}}$, Proposition~\ref{prop:constructionsofQI} shows that there is a quasi-isometry 
\begin{equation*}
    \varphi\colon\Z^{d_{1}}\wr G\longrightarrow \Z^{d_{2}}\wr H.
\end{equation*}
We therefore have our conclusion once we have proved the following: 
\begin{claim}\label{claim5.4}
An infinite virtually abelian group of growth degree $d\ge 1$ is biLipschitz equivalent to $\Z^d$. 
\end{claim}

{
\renewcommand{\proofname}{Proof of Claim~\ref{claim5.4}.}
\renewcommand{\qedsymbol}{$\blacksquare$}
\begin{proof}
If $A$ is virtually abelian and has growth degree $d$, it contains $\Z^d$ as a finite-index subgroup, and thus there is a quasi-$[A:\Z^d]$-to-one quasi-isometry $f\colon A\longrightarrow \Z^d$ (any quasi-inverse of the inclusion $\Z^d \hookrightarrow A$). Since $\text{Sc}(\Z^d)=\R_{>0}$, we may fix an arbitrary quasi-$\frac{1}{[A:\Z^d]}$-to-one quasi-isometry $g\colon \Z^d\longrightarrow \Z^d$, and the composition $g\circ f\colon A\longrightarrow \Z^d$ is therefore quasi-one-to-one according to Theorem~\ref{thm:stabilitypropertiesofscalingfactors}\textit{(ii)}. Thus, from Theorem~\ref{thm:whytetheorem}, we deduce that it lies at finite distance from a bijection, which is the desired biLipschitz equivalence $A\longrightarrow \Z^d$.
\end{proof}
\renewcommand{\qedsymbol}{$\square$}
}

Now, Claim~\ref{claim5.4} shows that $A_{1}$ is biLipschitz equivalent to $\Z^{d_{1}}$, so~\cite[Lemma~1]{Dyu00} implies that there is a biLipschitz equivalence $f_{1}\colon A_{1}\wr G\longrightarrow \Z^{d_{1}}\wr G$. Likewise, there is a biLipschitz equivalence $f_{2}\colon \Z^{d_{2}}\wr H\longrightarrow A_{2}\wr H$. The composition 
\begin{equation*}
    f_{2}\circ\varphi\circ f_{1}\colon A_{1}\wr G\longrightarrow A_{2}\wr H
\end{equation*}
is the quasi-isometry we are looking for. 

\smallskip

Let us now focus on \textit{(ii)}. The left-to-right direction is a particular case of Corollary~\ref{cor:BLE}. Conversely, if $d_{1}=d_{2}$, it follows from Claim~\ref{claim5.4} that $A_{1}$ and $A_{2}$ are biLipschitz equivalent, so~\cite[Lemma~1]{Dyu00} implies that $A_{1}\wr G$ and $A_{2}\wr G$ are biLipschitz equivalent. Additionally, Lemma~\ref{lm:BLEbetweenwreathproducts} ensures that $A_{2}\wr G$ and $A_{2}\wr H$ are biLipschitz equivalent, and composing these two biLipschitz equivalences provides a biLipschitz equivalence from $A_{1}\wr G$ to $A_{2}\wr H$.
\end{proof}}

{
\renewcommand{\proofname}{Proof of Corollary~\ref{cor:classificationforvirtuallyabeliangroups2}}
\begin{proof}
Assume that $A_{1}\wr B_{1}$ is quasi-isometric to $A_{2}\wr B_{2}$. In particular, F\o lner functions of these two groups must coincide up to $\simeq$, which, by~\cite[Theorem~1]{Ers03}, implies that 
\begin{equation*}
    \left(n^{\text{deg}(A_{1})}\right)^{n^{\text{deg}(B_{1})}} \simeq  \left(n^{\text{deg}(A_{2})}\right)^{n^{\text{deg}(B_{2})}}.
\end{equation*}
It follows that $n^{\text{deg}(B_{1})}\ln(n) \simeq n^{\text{deg}(B_{2})}\ln(n)$, and thus $n^{\text{deg}(B_{1})} \simeq n^{\text{deg}(B_{2})}$. The latter implies $\text{deg}(B_{1})=\text{deg}(B_{2})$ as claimed. 

\smallskip

Conversely, if $B_{1}$ and $B_{2}$ both have growth degree $d$, they are both biLipschitz equivalent to $\Z^{d}$ by Claim~\ref{claim5.4}, 
and since $\text{Sc}(\Z^d)=\R_{>0}$, we may fix an arbitrary quasi-$\frac{\text{deg}(A_{2})}{\text{deg}(A_{1})}$-to-one quasi-isometry $\Z^d\longrightarrow\Z^d$. The latter then provides us a quasi-$\frac{\text{deg}(A_{2})}{\text{deg}(A_{1})}$-to-one quasi-isometry $B_{1}\longrightarrow B_{2}$. Additionally, as there is a biLipschitz equivalence $A_{1}^{\text{deg}(A_{2})}\longrightarrow A_{2}^{\text{deg}(A_{1})}$, we conclude with Proposition~\ref{prop:constructionsofQI} that there is a quasi-isometry
\begin{equation*}
    A_{1}\wr B_{1} \longrightarrow A_{2}\wr B_{2}
\end{equation*}
as desired. The proof is complete. 
\end{proof}}

Concerning iterated wreath products, we prove first Proposition~\ref{prop:mixingofscalingconditions}, and Proposition~\ref{prop:lamplighterrigidity} follows immediately.

{
\renewcommand{\proofname}{Proof of Proposition~\ref{prop:mixingofscalingconditions}}
\begin{proof}
Assume that there exists a quasi-isometry 
\begin{equation*}
    \Z_{n}\wr(N_{1}\wr G)\longrightarrow \Z_{m}\wr(N_{2}\wr H).
\end{equation*}
From~\cite[Proposition~1.3]{GT24a}, $N_{1}\wr G$ and $N_{2}\wr H$ have the thick bigon property, and they are both amenable, so we may apply~\cite[Theorem~8.6]{GT24a} to deduce that there exist $a, r, s\ge 1$ such that $n=a^{r}$, $m=a^{s}$ and a quasi-$\frac{s}{r}$-to-one quasi-isometry $N_{1}\wr G\longrightarrow N_{2}\wr H$. By Theorem ~\ref{thm:maintheorem}, the latter is also quasi-$\frac{n_{2}}{n_{1}}$-to-one, and thus Lemma~\ref{lm:uniquenessofscalingfactors} forces $\frac{s}{r}=\frac{n_{2}}{n_{1}}$.
\end{proof}
}

We conclude with the proof of Corollary~\ref{cor:classificationofiteratedwreathproducts}.

{
\renewcommand{\proofname}{Proof of Corollary~\ref{cor:classificationofiteratedwreathproducts}}
\begin{proof}\textit{(i)} The left-to-right direction is a particular case of Proposition~\ref{prop:mixingofscalingconditions}. Conversely, given a quasi-$\frac{d_{2}}{d_{1}}$-to-one quasi-isometry $G\longrightarrow H$, we apply Proposition~\ref{prop:constructionsofQI} to construct an aptolic quasi-isometry
\begin{equation*}
    A_{1}\wr G\longrightarrow A_{2}\wr H.
\end{equation*}
From the proof of Proposition~\ref{prop:constructionsofQI} and Lemma~\ref{lm:scalingfactorsforwreathproducts}, this quasi-isometry is quasi-$\frac{d_{2}}{d_{1}}$-to-one, thus also quasi-$\frac{s}{r}$-to-one by assumption. As $n=a^{r}$ and $m=a^{s}$, we apply~\cite[Theorem~3.11]{GT24b} to deduce that there is a quasi-isometry
\begin{equation*}
    \Z_{n}\wr(A_{1}\wr G) \longrightarrow \Z_{m}\wr(A_{2}\wr H)
\end{equation*}
as claimed. 

\smallskip

Lastly, \textit{(ii)} follows from a combination of \textit{(i)} and of Lemma~\ref{lm:uniquenessofscalingfactors}.
\end{proof}}

\section{Comments and questions}\label{section6}

Results from~\cite{BGT24} still apply when lamp groups have superpolynomial but subexponential growth. For such groups, the growth function can exhibit several behaviors, and our understanding is much more limited, but from the proof of Theorem~\ref{thm:maintheorem} we can deduce the following:
\begin{corollary}\label{cor:intermediategrowth}
Let $\Gamma$ and $\Lambda$ be finitely generated groups with subexponential and superpolynomial growths. Let $G$ and $H$ be finitely presented groups from $\mathcal{M}_{\text{exp}}$. 

\begin{enumerate}[label=(\roman*)]
    \item If $\gamma_{\Gamma}(n)\simeq e^{n^{\alpha}}$, $\gamma_{\Lambda}(n)\simeq e^{n^{\beta}}$ for some $\alpha,\beta\in (0,1)$, and if \;$\Gamma\wr G$ and $\Lambda\wr H$ are quasi-isometric, then $\alpha=\beta$.
    \item If $\gamma_{\Gamma}(n)\simeq e^{\ln(n)n^{\alpha}}$, $\gamma_{\Lambda}(n)\simeq e^{\ln(n)n^{\beta}}$ for some $\alpha,\beta\in (0,1)$, and if \;$\Gamma\wr G$ and $\Lambda\wr H$ are quasi-isometric, then $\alpha=\beta$.
    \item If $\gamma_{\Gamma}(n)\simeq e^{\frac{n}{\ln^{\circ k}(n)}}$, $\gamma_{\Lambda}(n)\simeq e^{\frac{n}{\ln^{\circ d}(n)}}$ for some integers $d,k\ge 1$, and if \;$\Gamma\wr G$ and $\Lambda\wr H$ are quasi-isometric, then $k=d$.
\end{enumerate}
\end{corollary}

Here, $\gamma_{G}$ stands for the growth function of a group $G$, and for $k\ge 1$, $\ln^{\circ k}=\ln(\ln(\dots))$ denotes the $k-$th iteration of the log with itself.

\begin{proof}
The proof of Claim~\ref{claim4.2} in Section \ref{section4} in fact shows that some powers of $\gamma_{\Gamma}$ are equivalent to some powers of $\gamma_{\Lambda}$. More precisely, given our $(C,K)-$quasi-isometry 
\begin{equation*}
    \Gamma\wr G \longrightarrow \Lambda\wr H, \; (c,p)\longmapsto (\alpha(c),\beta(p)),
\end{equation*}
one has $\gamma_{\Gamma}^{|A|}(n) \simeq \gamma_{\Lambda}^{|\overline{\beta}(\beta(A)^{+K})^{+K}|}(n)$ for any finite subset $A\subset G$. A direct computation then shows that, for the three possible behaviours, such equivalences force equalities of the parameters. 
\end{proof}

Additionally, one also deduces from the equivalences
\begin{equation*}
    \gamma_{\Gamma}^{|A|}(n) \simeq \gamma_{\Lambda}^{|\overline{\beta}(\beta(A)^{+K})^{+K}|}(n),\; A\subset G \;\text{finite}
\end{equation*}
that these three ranges of behaviours are quasi-isometrically distinct when taking wreath products: for instance, if $\gamma_{\Gamma}(n)\simeq e^{n^{\alpha}}$ and $\gamma_{\Lambda}(n)\simeq e^{\ln(n)n^{\beta}}$ for some $\alpha,\beta\in (0,1)$ and if $G,H$ are finitely presented groups from $\mathcal{M}_{\text{exp}}$, then $\Gamma\wr G$ and $\Lambda\wr H$ are not quasi-isometric.

\smallskip

Moreover, we know that these three types of behaviour occur for finitely generated groups. It is shown in~\cite[Theorem~B]{EZ18} that some periodic Grigorchuk groups, among which the first Grigorchuk group, fall into the first class, and the logarithmic growth exponent is computed explicitly. Other examples appear in~\cite[Theorem~1]{BE12}, where the authors construct two families of groups $(K_{k})_{k\in\N}$ and $(H_{k})_{k\in\N}$ such that 
\begin{equation*}
    \gamma_{K_{k}}(n)\simeq e^{n^{1-(1-\alpha)^{k}}}\; \text{and}\; \gamma_{H_{k}}(n)\simeq e^{\ln(n)n^{1-(1-\alpha)^{k}}}
\end{equation*}
where $\alpha \cong 0.7674$ is a fixed constant. Thus, for instance, it follows from Corollary~\ref{cor:intermediategrowth} that, for $n\ge 2$, $K_{r}\wr \text{BS}(1,n)$ and $K_{s}\wr \text{BS}(1,n)$ are quasi-isometric if and only if $r=s$. 

\smallskip

Lastly, the existence of groups with growth functions asymptotically equal to $e^{\frac{n}{\ln^{\circ k}(n)}}$ for integers $k\ge 1$ is proved in~\cite[Theorem~A]{BE14}. 

\smallskip

However, in the case of intermediate growth, the fact that powers of the growth functions must be equivalent does not provide any valuable informations on scaling properties for $\beta$. Thus we may ask:
\begin{question}
Let $\Gamma$ be an intermediate growth group, and let $G$ be a finitely presented amenable group from $\mathcal{M}_{\text{exp}}$. Is $\text{Sc}(\Gamma\wr G)$ reduced to $\left\lbrace 1\right\rbrace$?
\end{question}

Note that the proof of Proposition~\ref{prop:lamplighterrigidity} only uses the fact that $N\wr G$ has the thick bigon property and a trivial scaling group. Therefore, any group with these two properties is lamplighter-rigid. 

\begin{corollary}
Let $n,m\ge 2$ be two integers. Let $H$ be a finitely generated amenable group satisfying the thick bigon property. If $\text{Sc}(H)=\lbrace 1\rbrace$, then $\Z_{n}\wr H$ and $\Z_{m}\wr H$ are quasi-isometric if and only if $n=m$. 
\end{corollary}

It would be interesting to find other examples of classes of groups satisfying these two properties. Plausible candidates could be higher-rank lamplighter groups used in~\cite{DPT15}, or direct products of standard wreath products. Hence:

\begin{question}
Let $d,q\ge 2$, and let $\Gamma_{d}(q)$ be the higher rank lamplighter group constructed in~\cite{DPT15}. Is $\text{Sc}(\Gamma_{d}(q))$ reduced to $\lbrace 1\rbrace$?
\end{question}

On the other hand, wreath products $F\wr K$ with $F$ finite and $K$ amenable finitely presented and one-ended, do not satisfy the thick bigon property, but still have trivial scaling group~\cite[proposition 6.8]{GT22}. Hence:

\begin{question}
Let $F$ be a non-trivial finite group, and let $K$ be a finitely presented one-ended amenable group. Is $F\wr K$ lamplighter-rigid?
\end{question}

Even though we focused on wreath products, it is worth noticing that scaling quasi-isometries also parametrize the existence of quasi-isometries between other \textit{halo products}, as defined in~\cite{GT24a}. Among these, \textit{lampshuffler groups} are of great interest, and, for instance, combining~\cite[Corollary~8.9]{GT24a} with Theorem~\ref{thm:classificationforvirtuallyabeliangroups}, it follows that $\shuf{\Z^n\wr G}$ and $\shuf{\Z^m\wr H}$ are quasi-isometric if and only if $n=m$, when $G$ and $H$ are amenable finitely presented groups from $\mathcal{M}_{\text{exp}}$. 

\smallskip

Finally, regarding quasi-isometric classifications of wreath products, it is natural to wonder whether there is an iterated version of Corollary~\ref{cor:classificationforvirtuallyabeliangroups2}, at least for wreath products of the form 
\begin{equation*}
\Z^{n_{1}}\wr(\Z^{n_{2}}\wr\dots\wr(\Z^{n_{r}}\wr\Z^{k}))),\; n_{1},\dots,n_{r},k\ge 1. 
\end{equation*}
The existence of a quasi-isometry 
\begin{equation*}
\Z^{n_{1}}\wr(\Z^{n_{2}}\wr(\dots\wr(\Z^{n_{r}}\wr\Z^{k}))) \longrightarrow \Z^{m_{1}}\wr(\Z^{m_{2}}\wr(\dots\wr(\Z^{m_{r}}\wr\Z^{k'})))
\end{equation*}
already imposes $k=k'$, by computing F\o lner functions, and it seems not unreasonable to believe that it imposes an arithmetic condition of the kind $\frac{m_{1}}{n_{1}}=\frac{m_{2}}{n_{2}}=\dots=\frac{m_{r}}{n_{r}}$, in the spirit of Proposition~\ref{prop:mixingofscalingconditions}. Additionally, note that Proposition~\ref{prop:constructionsofQI} already provides the flexibility part, if all these quotients are equal. Hence:
\begin{question}
Let $n_{1},m_{1},\dots,n_{r},m_{r}\ge 1$ and $k,k'\ge 1$. Is it true that $\Z^{n_{1}}\wr(\Z^{n_{2}}\wr(\dots\wr(\Z^{n_{r}}\wr\Z^{k})))$ and $\Z^{m_{1}}\wr(\Z^{m_{2}}\wr(\dots\wr(\Z^{m_{r}}\wr\Z^{k'})))$ are quasi-isometric if and only if $\frac{m_{1}}{n_{1}}=\frac{m_{2}}{n_{2}}=\dots=\frac{m_{r}}{n_{r}}$ and $k=k'$?
\end{question}

\printbibliography

{\bigskip
		\footnotesize
		
		\noindent V.~Dumoncel, \textsc{Université Paris Cité, Institut de Mathématiques de Jussieu-Paris Rive Gauche, 75013 Paris, France}\par\nopagebreak\noindent
		\textit{E-mail address: }\texttt{vincent.dumoncel@imj-prg.fr}}

\end{titlepage}
\end{document}